\documentclass[12pt]{article}
\usepackage{amsfonts}
\usepackage{stmaryrd}
\usepackage{graphicx}
\usepackage{epsfig}
\usepackage{xcolor}
\usepackage{amsmath}
\usepackage{amsthm}
\usepackage{amssymb}
\usepackage{mathrsfs}
\usepackage{float}
\usepackage{multirow}
\usepackage{verbatim}
\usepackage{fancyhdr}
\usepackage{subfigure}
\usepackage{color}
\usepackage{mathtools}
\usepackage{mathrsfs}
\usepackage{sectsty}
\usepackage[title]{appendix}
\usepackage{threeparttable}
\usepackage{dcolumn}
\usepackage{booktabs}
\usepackage{indentfirst}
\usepackage{setspace}
\usepackage{bm}
\usepackage{enumerate}
\usepackage{lineno}
\usepackage{bbm}
\usepackage[labelfont=bf,labelsep=period]{caption}
\newcolumntype{d}[1]{D{.}{.}{#1}}
\usepackage{hyperref}
\hypersetup{colorlinks,linkcolor=blue,citecolor=blue}
\newtheorem{Def}{Definition}[section]
\newtheorem{Thm}{Theorem}[section]
\newtheorem{Ex}{Example}[section]

\newtheorem{Remark}{Remark}[section]

\allowdisplaybreaks

\begin{document}

\pagenumbering{arabic}
\baselineskip=1.3pc

\vspace*{0.5in}

\begin{center}

{{\bf Structure-preserving nodal DG method for Euler equations with gravity II: general equilibrium states}}

\end{center}

\vspace{.03in}

\centerline{
Yuchang Liu\footnote{School of Mathematical Sciences,
         University of Science and Technology of China,
         Hefei, Anhui 230026, P.R. China.  
         E-mail: lissandra@mail.ustc.edu.cn.},
Wei Guo\footnote{Department of Mathematics and Statistics, 
	Texas Tech University, Lubbock, TX, 70409, USA. 
	E-mail: weimath.guo@ttu.edu.},
Yan Jiang\footnote{School of Mathematical Sciences,
         University of Science and Technology of China, Hefei,
         Anhui 230026, P.R. China.  
         E-mail: jiangy@ustc.edu.cn.
         Research supported by NSFC grant 12271499. },
and Mengping Zhang\footnote{School of Mathematical Sciences,
         University of Science and Technology of China, Hefei,
         Anhui 230026, P.R. China.  
         E-mail: mpzhang@ustc.edu.cn.}
}

\vspace{.1in}

\noindent
{\bf Abstract:} We develop an entropy-stable nodal discontinuous Galerkin (DG) scheme for the Euler equations with gravity, which is also well-balanced with respect to general equilibrium solutions, including both hydrostatic and moving equilibria. The core of our approach lies in a novel treatment of the gravitational source term, combining entropy-conservative numerical fluxes with a linear entropy correction. In addition, the proposed formulation is carefully designed to ensure compatibility with a positivity-preserving limiter. We provide a rigorous theoretical analysis to establish the accuracy and structure-preserving properties of the proposed scheme.  
Extensive numerical experiments confirm the robustness and efficiency of the scheme.

\vspace{.1in}


%\tableofcontents

\section{Introduction}

In this paper, we continue our investigation into the development of accurate, theoretically justified, and structure-preserving numerical schemes for the Euler equations with gravity, building on our previous work \cite{liu2025structure}. These equations form a nonlinear system of balance laws with gravitational source terms, and arise naturally in various physical applications, particularly in astrophysics and atmospheric science. It is well recognized that preserving key properties of the underlying balance law at the discrete level is critical for avoiding non-physical behavior and simulation failure \cite{balsara2004comparison, chen2017entropy, zhang2010positivity}. In recent decades, substantial research efforts have been devoted to the development of structure-preserving methods, which continue to pose fundamental challenges in algorithm design \cite{slotnick2014cfd}. In \cite{liu2025structure}, we employed the discontinuous Galerkin (DG) method \cite{reed1973triangular, cockburn1989tvb, cockburn1989tvb3, cockburn1990runge2, cockburn1991runge} as a building block, leveraging its distinctive properties, such as high-order accuracy, local conservation,  and flexibility to design new and incorporate existing well-established structure preserving techniques.
In particular, we focused on three key properties, namely well-balanceness (WB), entropy stability (ES), and the positivity preserving (PP) property for the Euler equations with gravity. 
By carefully modifying the source term to match the form of entropy conservative fluxes \cite{chen2017entropy, liu2018entropy}, and incorporating the PP limiter \cite{zhang2010positivity, zhang2011positivity}, we developed a novel nodal DG scheme that is WB in \emph{capturing arbitrary hydrostatic equilibrium}, and ES and PP for general solutions. This is the first DG scheme in the literature that can simultaneously achieve all three properties.

Meanwhile, the WB property of the method in \cite{liu2025structure}, as with many other existing WB schemes \cite{xing2013high, chandrashekar2015second, li2016well, chandrashekar2017well, wu2021uniformly, du2024well}, is restricted to hydrostatic equilibrium solutions. To date, there have been relatively few studies on the WB property for capturing more general moving equilibrium states, partly due to the additional complexity arising from nonlinear interactions introduced by nonzero velocity fields. In \cite{gomez2021high}, a control-based approach is introduced for general balance laws, where the nonlinear reconstruction problem is formulated as a control problem. In \cite{grosheintz2020well}, a WB finite volume method was proposed by using a local steady state reconstruction to preserve numerical moving equilibrium states which works in Cartesian, cylindrical, and spherical coordinates. In \cite{berthon2024entropy}, Berthon \emph{et al.} introduced a structure preserving finite volume scheme in one dimension 
(1D), which is WB for general equilibrium states, PP, and ES for the first-order version. This is achieved by carefully designing an approximate Riemann solver at interfaces. For DG discretizations, Zhang \textit{et al.} \cite{zhang2024equilibrium} proposed a WB DG method that preserves general equilibrium states by updating equilibrium variables instead of conservative variables. In \cite{xu2024high}, the DG method exploits the interpolatory property of Gauss–Lobatto points to carefully design the discretization of the source term. Both methods achieve the WB property for general equilibrium states of the shallow water equations in both 1D and two dimensions (2D), as well as for general equilibrium states of the Euler equations with gravity in 1D.

In this work, motivated by \cite{xu2024high}, we extend our previous scheme in \cite{liu2025structure} to simultaneously achieve ES and WB properties for arbitrary equilibrium state solutions. In particular, the form of the equilibrium solution needs to be specified \textit{a priori}, which is typically required in the construction of WB schemes. The proposed method relies on two key ingredients. First, following the approach in \cite{liu2025structure}, by adjusting the source term discretization to exactly match the form of entropy conservative fluxes, we obtain the WB property. Note that, for hydrostatic equilibrium, this is shown to be sufficient to ensure entropy stability, as the entropy variable is orthogonal to the source term in the zero-velocity setting, see \cite{liu2025structure}. However, the orthogonality no longer holds for moving equilibrium states. To address this, inspired by \cite{abgrall2018general, gaburro2023high, liu2024non}, we incorporate a conservative, high order entropy correction term into the DG formulation to control the additional entropy produced by the source term, without destroying the WB property. Note that both techniques are essential: if we only employ the entropy conservative fluxes, then the discretization of the source term may increase the entropy; if only the entropy correction term is added, then the WB property cannot be guaranteed. Furthermore, the proposed modification is shown to be compatible with the PP limiter \cite{zhang2011positivity}. A specially designed set of numerical experiments verifies that omitting any single property would lead to computational breakdown or violation of key physical laws, highlighting the effectiveness of the proposed structure-preserving method. We also provide a rigorous theoretical analysis to establish all these properties.

The rest of this paper is organized as follows. In Section \ref{sec2}, we introduce the Euler equations with gravity, together with their equilibrium solutions and entropy condition. In Section \ref{sec3}, we present the structure-preserving nodal DG scheme in 1D, followed by its extension to the 2D case in Section~\ref{sec4}. Section~\ref{sec5} provides a variety of numerical examples to demonstrate the performance of the proposed scheme. Concluding remarks are given in Section~\ref{sec6}.

%%%%%%%%%%%%%%%%%%%%%%%%%%%%%%%%%%%%%%%%%%%%%%%%%
\section{Euler equations with gravity}\label{sec2}

\subsection{Governing equations}

In general $d$-dimensional space, the Euler equations with gravity can be written as a system of nonlinear balance law
\begin{equation}\label{eq:Euler-d}
\mathbf{U}_t+\nabla \cdot \mathbf{F}\left( \mathbf{U} \right) =\mathbf{S}\left( \mathbf{U},\mathbf x \right),\quad (\mathbf x,t)\in \mathbb R^d\times [0,+\infty),
\end{equation}
where
\begin{equation}
\mathbf{U}=\left[ \begin{array}{c}
	\rho\\
	\mathbf{m}\\
	\mathcal{E}\\
\end{array} \right] ,\quad \mathbf{F}\left( \mathbf{U} \right) =\left[ \begin{array}{c}
	\mathbf m\\
	\rho \mathbf{u}\otimes \mathbf{u}+pI_d\\
	\mathbf{u}\left( \mathcal{E} +p \right)\\
\end{array} \right] ,\quad \mathbf{S}\left( \mathbf{U},\mathbf x \right) =\left[ \begin{array}{c}
	0\\
	-\rho \nabla \phi\\
	- \mathbf{m}\cdot \nabla \phi\\
\end{array} \right].
\end{equation}
The variables are defined as follows: $\rho$ denotes the mass density, $\mathbf m=\rho\mathbf u$ represents the momentum, $\mathcal E$ characterizes the total energy, $I_d$ denotes the identity matrix, and $p$ represents the thermodynamic pressure. The source term $\mathbf S$ encapsulates the gravitational field effects, where $\phi=\phi(\mathbf x)$ denotes the time-invariant gravitational potential. The system closure is achieved through the implementation of the ideal gas equation of state
\begin{equation}\label{eq:EOS}
\mathcal{E} =\frac{p}{\gamma -1}+\frac{1}{2}\rho \left\|\mathbf u\right\|^2,\quad \gamma =1.4.
\end{equation}
A necessary condition for the existence of physically admissible solutions to \eqref{eq:Euler-d} is that the solution must lie in the admissible state space defined by 
$$\mathbf U(x,t)\in\mathscr G,\quad\forall (x,t)\in\mathbb R^d\times [0,+\infty),$$
where the \textit{admissible set}
$$ \mathscr G=\left\{ (\rho,\mathbf m,\mathcal E):\ \rho > 0\quad\mathrm{and}\quad p(\rho,\mathbf m,\mathcal E)>0 \right\}. $$
It is known that the set $\mathscr G$ is convex when $\rho > 0$, a property that is fundamental to the analysis of positivity preservation.

\subsection{Steady state solutions}

In the presence of the gravitational potential $\phi$, the system \eqref{eq:Euler-d} admits a class of time-independent solutions, designated as \textit{equilibrium state solutions}, which satisfies
\begin{equation}\label{eq:eqbm}\nabla\cdot\mathbf F(\mathbf U)=\mathbf S(\mathbf U,\mathbf x). \end{equation}
In previous investigations \cite{xing2013high, chandrashekar2017well, li2016well, du2024well}, researchers have predominantly focused on \textit{hydrostatic equilibrium states} characterized by vanishing velocities. For such equilibrium states, \eqref{eq:eqbm} reduces to
\begin{equation}\label{eq:hydroeqbm}
\rho = \rho(\mathbf{x}), \quad \mathbf{u} = \mathbf{0}, \quad \nabla p = -\rho \nabla \phi.
\end{equation}
Hydrostatic equilibrium solutions fall into two primary categories: isothermal and isentropic states. In isothermal equilibrium, temperature remains constant at $T_0$, yielding the equation of state 
$p/\rho=RT_0$, where $R$ denotes the gas constant. The corresponding solution is given by \begin{equation}\label{eq:isT}
\rho = \rho_0 \exp \left( -\frac{\phi}{RT_0} \right), \quad \mathbf{u} = \mathbf{0}, \quad p = \rho_0RT_0 \exp \left( -\frac{\phi}{RT_0} \right).
\end{equation}
The isentropic equilibrium state assumes $p\rho^{-\gamma}=K_0$, and the solution is given by
\begin{equation}\label{eq:isE}
\rho = \left( \frac{\gamma - 1}{K_0 \gamma} \left( C - \phi \right) \right)^{\frac{1}{\gamma - 1}}, \quad \mathbf{u} = \mathbf{0}, \quad p = K_0 \rho^\gamma.
\end{equation}
Here, $\rho_0,R,T_0,K_0,C_0$ are constants.  
Note that they can both be regarded as more general polytropic equilibrium states, which  assumes the relation $p\rho^{-\nu}=K_0$, where $\nu$ and $K_0$ are constants. When $\nu = 1$, it recovers the isothermal equilibrium; when $\nu = \gamma$, it recovers the isentropic equilibrium.

A more nontrivial case arises in the form of equilibrium states with non-vanishing velocities $\mathbf u\ne\mathbf 0$, termed \textit{moving equilibrium states}, which have received relatively limited attention in the literature. In the 1D case, it is shown that 1D moving equilibrium state solution $(\rho,m,\mathcal E)$ is necessarily isentropic \cite{berthon2024entropy}. Specifically,
\begin{equation}\label{eq:movingeqbm}
m=m_0,\quad H+\phi=\frac{\mathcal E+p}{\rho}+\phi =H_0,\quad s=\ln \left( p\rho ^{-\gamma} \right) =s_0,
\end{equation}
where $m_0,\ H_0,\ s_0$ are constants. Based on this property, in \cite{berthon2024entropy}, the authors carefully designed a WB finite volume scheme for 1D moving equilibrium states. However, in higher dimensions, the relation \eqref{eq:movingeqbm} is no longer globally valid, but rather holds only along streamlines. Consequently, the development of novel methodologies for multidimensional moving equilibrium states presents nontrivial challenges.

\subsection{Entropy Condition}

For a $d$-dimensional system of conservation laws:
\begin{equation}\label{eq:hcl}
\mathbf{U}_t + \nabla \cdot \mathbf{F}(\mathbf{U}) = \mathbf{0},\quad \mathbf{F} = (\mathbf{F}_1, \dots, \mathbf{F}_d),
\end{equation}
a convex function $\mathcal{U}(\mathbf{U})$ is designated as an entropy function for the system \eqref{eq:hcl} if there exists an associated entropy flux vector $\mathcal{F} = (\mathcal{F}_1, \dots, \mathcal{F}_d)$ satisfying
\begin{equation}\label{eq:entropy}
\mathcal{F}_i'(\mathbf{U}) = \mathcal{U}'(\mathbf{U}) \,\mathbf{F}_i'(\mathbf{U}), \quad i = 1, \dots, d.
\end{equation}
Here, $\mathcal F_i'(\mathbf U)$ and $\mathcal U'(\mathbf U)$ are viewed as row vectors.  The pair of functions $(\mathcal{U}, \mathcal{F})$ constitutes an \textit{entropy pair}, with the associated \textit{entropy variable} defined as $\mathbf{V} = \mathcal{U}'(\mathbf{U})^T$. For a system of conservation laws \eqref{eq:hcl} admitting an entropy pair, left-multiplying $\mathbf{V}(\mathbf{U})^T$ yields an additional relation
\begin{equation}\label{eq:EC}
\mathcal{U}(\mathbf{U})_t + \nabla \cdot \mathcal{F}(\mathbf{U}) = 0
\end{equation}
for smooth solutions. For non-smooth solutions, the above equation is replaced by an inequality:
\begin{equation}\label{eq:ES}
\mathcal{U}(\mathbf{U})_t + \nabla \cdot \mathcal{F}(\mathbf{U}) \le 0,
\end{equation}
which holds in the weak sense and is known as the \textit{entropy condition}. A strictly convex function $\mathcal{U}$ constitutes an entropy function if and only if the Jacobian matrix $\partial \mathbf{U} / \partial \mathbf{V}$ is symmetric positive-definite and the matrices $\partial \mathbf{F}_i(\mathbf{U}(\mathbf{V})) / \partial \mathbf{V}$ are symmetric for all $i$. 
When these conditions are met, it is possible to define scalar functions $\varphi(\mathbf{V})$ and $\psi_i(\mathbf{V})$, with continuous second derivatives, termed the \textit{potential function} and \textit{potential flux} respectively, satisfying the specified relations
\begin{equation}\label{eq:psi}
\mathbf{U}(\mathbf{V})^T = \frac{\partial \varphi}{\partial \mathbf{V}}, \quad \mathbf{F}_i(\mathbf{V})^T = \frac{\partial \psi_i}{\partial \mathbf{V}}.
\end{equation}
For the homogeneous Euler equations, i.e. \eqref{eq:Euler-d} without the source term $\mathbf{S}(\mathbf{U}, \mathbf{x})$, Harten \cite{HARTEN1983151} established the existence of a family of entropy pairs parameterized by the specific entropy $s = \ln(p \rho^{-\gamma})$, which satisfy the symmetric conditions. 

However, to symmetrize the viscous terms in the compressible Navier-Stokes equations with thermal conductivity \cite{HUGHES1986223}, there exists a unique entropy pair:
\begin{equation}\label{eq:entropy_pair}
    \mathcal{U} = -\frac{\rho s}{\gamma - 1}, \quad \mathcal{F} = -\frac{\rho s}{\gamma - 1}\mathbf u.
\end{equation}
The corresponding entropy variable is given by
\begin{equation}\label{eq:entropy_variable}
\mathbf{V} = \mathcal{U}'(\mathbf{U})^T = \left[ \begin{array}{c}
    \dfrac{\gamma - s}{\gamma - 1} - \dfrac{\rho \|\mathbf{u}\|^2}{2p} \\
    {\rho \mathbf{u}}/{p} \\
    -{\rho}/{p} \\
\end{array} \right],    
\end{equation}
with potential functions $\varphi = \rho$ and $\psi_i = \rho u_i$. Direct verification shows that the Euler system with gravitational terms \eqref{eq:Euler-d} satisfies the entropy condition \eqref{eq:ES} when employing the entropy pair \eqref{eq:entropy_pair} and entropy variable \eqref{eq:entropy_variable} \cite{desveaux2016well}, since
$$
\mathbf{V}(\mathbf{U})^T \mathbf{S}(\mathbf{U}, \mathbf{x}) = -\rho \nabla \phi \cdot \frac{\rho \mathbf{u}}{p} - \left( -\rho \mathbf{u} \cdot \nabla \phi \right) \frac{\rho}{p} = 0.
$$
Thus, in the following, we will consider the entropy pair \eqref{eq:entropy_pair} and entropy variable \eqref{eq:entropy_variable} for the Euler equations with gravity \eqref{eq:Euler-d}.

%%%%%%%%%%%%%%%%%%%%%%%%%%%%%%%%%%%%%%%%%%%%%%%%%%%%%
\section{Structure preserving nodal DG method in one dimension}
\label{sec3}

We begin from the 1D case, and the Euler system \eqref{eq:Euler-d} with gravity can be written as:
\begin{equation}\label{eq:Euler1D}
\left[ \begin{array}{c}
    \rho \\
    m \\
    \mathcal{E} \\
\end{array} \right]_t + \left[ \begin{array}{c}
    m \\
    \rho u^2 + p \\
    u (\mathcal{E} + p) \\
\end{array} \right]_x = \left[ \begin{array}{c}
    0 \\
    -\rho \phi_x \\
    -m \phi_x \\
\end{array} \right],
\end{equation}
where $m = \rho u$ denotes the momentum density. The system is denoted by
\begin{equation}\label{eq:EulerG-1D2} 
\mathbf{U}_t + \mathbf{F}(\mathbf{U})_x = \mathbf{S}(\mathbf{U}, x). 
\end{equation}

We divide the 1D spatial domain $\Omega=[a,b]$ into $N$ cells $\mathcal K=\{K_{i}=[x_{i-1/2},x_{i+1/2}], i=1, \cdots, N\}$ with $x_{1/2}=a$, $x_{N+1/2}=b$. Without loss of generality, we assume a uniform grid with $\Delta x = \Delta x_i=x_{i+1/2}-x_{i-1/2}$ and denote the cell center $x_i=\frac{1}{2} (x_{i-1/2} + x_{i+1/2})$.

Firstly, we introduce the original nodal DG scheme for \eqref{eq:Euler1D}. Consider the reference element \(I = [-1, 1]\) associated with the Gauss–Lobatto quadrature points
$$
-1 = X_0 < X_1 < \cdots < X_k = 1,
$$
and the corresponding quadrature weights \(\{\omega_l\}_{l=0}^{k}\).
We define the differentiation matrix \(D\) with elements $D_{jl} = L_l'(X_j)$, where $L_l$ represents the $l$-th Lagrange basis function that satisfies the condition $L_l(X_j)=\delta_{lj}$. The mass matrix \(M\) and stiffness matrix \(S\) are expressed as
$$M = \text{diag}\{\omega_0, \omega_1, \dots, \omega_k\},\quad S= MD.$$
Let the boundary matrix \(B\) be defined as 
$$
B = \mathrm{diag}\{-1, 0, 0, \dots, 0, 0, 1\} =: \mathrm{diag}\{\tau_0, \dots, \tau_k\},
$$
then the \textit{summation-by-parts} (SBP) property holds \cite{chen2017entropy}:
\begin{equation}\label{eq:SBP1} 
S+S^T=B. 
\end{equation}
Moreover,
\begin{equation}\label{eq:SBP2}
  \sum_{l=0}^{k} D_{jl} = \sum_{l=0}^{k} S_{jl} = 0, \quad \sum_{l=0}^{k} S_{lj} = \tau_j,\quad \forall\, 0\le j\le k.  
\end{equation}
We consider the finite element space to be the collection of piecewise polynomial functions having degree no greater than $k$:
$$
V_h^k = \left\{ w(x): w(x) |_{K_i} \in P^k(K_i), \quad \forall K_i \in \mathcal{K} \right\},
$$
then the solution $\mathbf U_h\in\mathbf V_h^k$ can be represented by
\begin{align}
\mathbf{U}_h(x)|_{K_i} = \sum_{i_i=0}^{k} \mathbf{U}_h\left(x_{i} + \frac{\Delta x}{2} X\right) L_{i_1} \left( \frac{x-x_i}{\Delta x/2} \right)
\end{align}
on cell $K_i$, where $L_{i_1}$ denotes the $i_1$-th Lagrange interpolation polynomial on $K_0=[-1,1]$.  Utilizing the matrices established earlier, we can formulate the DG scheme using a compact matrix-vector representation based on nodal variables. We employ the following simplified notations:
\begin{align}
 x_i(X) = x_{i} + \frac{\Delta x}{2} X,\quad
 \mathbf{U}_{i_1}^{i} = \mathbf{U}_h(x_i(X_{i_1})), \quad \mathbf{F}_{i_1}^{i} = \mathbf{F}(\mathbf{U}_{i_1}^{i}),
\end{align}
\begin{equation}
\label{eq:Fstar}
\mathbf{F}_{i_1}^{*,i} = \begin{cases}
\hat{\mathbf F}(\mathbf U_k^{i-1},\mathbf U_0^{i})=:\hat{\mathbf{F}}_{i-1/2}, & i_1 = 0, \\
0, & 0 < i_1 < k, \\
\hat{\mathbf F}(\mathbf U_k^i,\mathbf U_0^{i+1})=:\hat{\mathbf{F}}_{i+1/2}, & i_1 = k.
\end{cases}
\end{equation}
Here, $\hat{\mathbf{F}}_{i+1/2}$ represents the numerical flux at the cell interface. In this paper, we use the Lax-Friedrichs flux:
\begin{equation}\label{eq:LF}
\hat{\mathbf{F}}^{LF}(\mathbf{U}_L, \mathbf{U}_R) = \frac{1}{2} (\mathbf{F}(\mathbf{U}_R) + \mathbf{F}(\mathbf{U}_L)) - \frac{\alpha}{2} (\mathbf{U}_R - \mathbf{U}_L).
\end{equation}
The parameter $\alpha$ corresponds to an estimated maximum local wave propagation speed, with its calculation to be specified in following discussions. 
The original nodal DG strong-formulation for \eqref{eq:EulerG-1D2} is given by \cite{chen2017entropy}
\begin{equation}\label{eq:nodal-old}
\frac{\mathrm{d}\mathbf{U}_{i_1}^{i}}{\mathrm{d}t} + \frac{2}{\Delta x}\sum_{l=0}^k D_{i_1, l} \mathbf{F}_l^i + \frac{2}{\Delta x}\frac{\tau_{i_1}}{\omega_{i_1}} \left( \mathbf{F}_{i_1}^{*,i} - \mathbf{F}_{i_1}^{i} \right) = \mathbf{S} \left( \mathbf{U}_{i_1}^{i}, x_i(X_{i_1}) \right).
\end{equation}

However, it is known that the nodal DG scheme above may not inherently preserve critical physical structures of interest. To address this limitation, we propose a modified DG framework that rigorously maintains the following fundamental properties simultaneously:

\begin{itemize}
\item[1.] \textbf{Well-balanced (WB) property}: 
If the initial condition is a steady-state solution  $\mathbf{U}^e(x) = (\rho^e(x), m^e(x), \mathcal{E}^e(x))^T$, the numerical solution remains \(\mathbf{U}_h = \mathbf{U}_h^e\). Here $\mathbf{U}_h^e$ is the interpolating polynomial of $\mathbf{U}^e$ at Gauss-Lobatto points. 
In particular, we concern on both the hydrostatic equilibrium states \eqref{eq:hydroeqbm} and moving equilibrium state \eqref{eq:movingeqbm}.

\item[2.] \textbf{Entropy-stable (ES) property}: Under the assumption of periodic or compactly supported boundary conditions, the semi-discrete scheme satisfies
$$
\frac{\mathrm{d} }{\mathrm{d}t} \left( \sum_{i=1}^N \sum_{i_1=0}^k \frac{\Delta x}{2} \omega_{i_1} 
 \mathcal{U}(\mathbf{U}_{i_1}^{i}) \right)  \le 0.
$$

\item[3.] \textbf{Positivity-preserving (PP) property}: If the numerical solution at time $t^n$ satisfies \(\mathbf{U}_{i_1}^{i, n} \in \mathscr{G}\), then the updated solution at \(t^{n+1}\) satisfies
$$
\mathbf{U}_{i_1}^{i, n+1} \in \mathscr{G}\quad 
\text{for any \(i\) and \(i_1\)}.
$$

\end{itemize}

\subsection{Proposed scheme}
Before introducing the proposed scheme, we define the entropy conservative flux $\mathbf F^S$ and the entropy stable flux $\hat{\mathbf F}$, respectively,  which play a key role in the algorithm development. 

\begin{Def}[Entropy conservative flux]
A consistent, symmetric two-point numerical flux $\mathbf F^S(\mathbf U_L, \mathbf U_R)$ is said to be entropy conservative, if for the given entropy function $\mathcal U$,
\begin{equation}\label{eq:ESflux}
\left( \mathbf{V}_R - \mathbf{V}_L \right) ^T \mathbf{F}^S \left( \mathbf{U}_L, \mathbf{U}_R \right) = \left( \psi_R - \psi_L \right).
\end{equation}
\end{Def}

\begin{Def}[Entropy stable flux]
A consistent two-point numerical flux $\hat{\mathbf F}(\mathbf U_L, \mathbf U_R)$ is said to be entropy stable, if for the given entropy function $\mathcal U$,
$$
\left( \mathbf{V}_R - \mathbf{V}_L \right) ^T \hat{\mathbf{F}} \left( \mathbf{U}_L, \mathbf{U}_R \right) \le \left( \psi_R - \psi_L \right).
$$
\end{Def}

For instance, \cite{chandrashekar2013kinetic} proposed the following entropy conservative flux for the Euler equations without gravity,
$$
\begin{aligned}
F_{1}^{S} &= \hat{\rho} \bar{u}, \\
F_{2}^{S} &= \frac{\bar{\rho}}{2\bar{\beta}} + \bar{u} F_{1}^{S}, \\
F_{3}^{S} &= \left( \frac{1}{2(\gamma - 1)\hat\beta} - \frac{1}{2} \overline{u^2} \right) F_1^S + \bar{u} F^2_S,
\end{aligned}
$$
where $\beta = \rho / 2p$, and 
$$
\overline{\alpha} = \frac{\alpha_l+\alpha_r}{2}, \quad \hat{\alpha} = \frac{\alpha_r - \alpha_l}{\ln \alpha_r - \ln \alpha_l}.
$$
In the context of entropy stable numerical fluxes, \cite{toro2013riemann} proposes utilizing the two-rarefaction approximation approach. Subsequently, \cite{guermond2016fast} proved that this two-rarefaction approximated wave speed methodology ensures entropy stability for the Euler system with $1 < \gamma \le 5/3$ when implemented the Lax-Friedrichs flux. Based on their similarity on the entropy property of Euler equations without or with gravity, we can employ those entropy conservative/stable fluxes there.  Moreover, to ensure entropy stability, we set the parameter $\alpha$ in the Lax-Friedrichs flux \eqref{eq:LF} as
$$
\alpha = \max \left\{ \left| u_L \right| + c_L, \left| u_R \right| + c_R, \alpha^{RRF}(\mathbf U_L, \mathbf U_R) \right\},
$$
where $c = \sqrt{\gamma p / \rho}$, and $\alpha^{RRF}$ is determined using the two-rarefaction approximation technique \cite{chen2017entropy}. \\

Next, we will give the proposed scheme. 
Note that for a given steady state $\mathbf U^e$ that satisfies \eqref{eq:eqbm}, we have
$\mathbf F(\mathbf U^e)_x=\mathbf S(\mathbf U^e,x).$ 
Following a similar technique in \cite{du2024well, xu2024high}, we can rewrite \eqref{eq:EulerG-1D2} as
\begin{equation}\label{eq:EulerG-1Dmod} 
\mathbf U_t+\mathbf F(\mathbf U)_x=\mathbf S(\mathbf U,x)+\mathbf F(\mathbf U^e)_x-\mathbf S(\mathbf U^e,x). 
\end{equation}
In this work, we design the scheme based on the above equation. Specially, to guarantee the WB property, the discretization for the source term should precisely match the flux term in the equilibrium state solution. 
Numerous previous works have utilized this concept to design numerical schemes. However, they are typically only well-balanced without satisfying entropy conditions, e.g. \cite{li2016well, xu2024high}. In our previous work on the structure-preserving scheme \cite{liu2025structure}, it only preserves the hydrostatic equilibrium and can not be generalized to arbitrary equilibrium states.

To address these issues, the novelly proposed nodal DG scheme is given by the ``strong" formulation: 
\begin{equation}\label{eq:scheme1D}
\frac{\mathrm{d}\mathbf{U}_{i_1}^{i}}{\mathrm{d}t}+\frac{2}{\Delta x}\sum_{l=0}^k{2D_{i_1,l}\mathbf{F}^S\left( \mathbf{U}_{i_1}^{i},\mathbf{U}_{l}^{i} \right)}+\frac{2}{\Delta x}\frac{\tau _{i_1}}{\omega _{i_1}}\left( \mathbf{F}_{i_1}^{*,i}-\mathbf{F}_{i_1}^{i} \right) =\mathbf{S}_{i_1}^{i}+(\mathbf{S}^0)^{i}_{i_1}-(\mathbf{S}^{corr})_{i_1}^i,
\end{equation}
where $\mathbf F_{i_1}^{*,i}$ is defined in \eqref{eq:Fstar}, and
$$
\mathbf{S}_{i_1}^{0,i}=\frac{2}{\Delta x}\sum_{l=0}^k{2D_{i_1,l}\mathbf{F}^S\left( \mathbf{U}_{i_1}^{e,i},\mathbf{U}_{l}^{e,i} \right)}-\mathbf S^{e,i}_{i_1},$$
$$\mathbf{S}_{i_1}^{e,i}=\left[ \begin{array}{c}
	0\\
	-\rho _{i_1}^{e,i} \phi _{x,i_1}^{i}\\
	-m_{i_1}^{e,i} \phi _{x,i_1}^{i}\\
\end{array} \right],\quad \mathbf{S}_{i_1}^{i}=\left[ \begin{array}{c}
	0\\
	-\rho _{i_1}^{i} \phi _{x,i_1}^{i}\\
	-m_{i_1}^{i} \phi _{x,i_1}^{i}\\
\end{array} \right] ,$$
$$\begin{aligned}\mathbf{S}_{i_1}^{corr,i}=\sigma _i\left( \mathbf{V}_{i_1}^{i}-\bar{\mathbf{V}}^i \right),&\quad 
\sigma _i=\frac{\displaystyle\sum_{i_1=0}^k\omega_{i_1}{\left( \mathbf{V}_{i_1}^{i}-\mathbf V_{i_1}^{e,i} \right) ^T\mathbf{S}_{i_1}^{0,i}}}{\displaystyle\sum_{i_1=0}^k\omega_{i_1}{\left\| \mathbf{V}_{i_1}^{i}-\bar{\mathbf{V}}^i \right\| ^2}},
\\
\mathbf V_{i_1}^i=\mathbf V(\mathbf U_{i_1}^i),\quad \mathbf V_{i_1}^{e,i}&=\mathbf V(\mathbf U_{i_1}^{e,i}),\quad \bar{\mathbf V}^i=\sum\limits_{i_1=0}^k\frac{\omega_{i_1}}{2}\mathbf V_{i_1}^i. 
\end{aligned}$$
In particular, if $\sum\limits_{i_1=0}^k\omega_{i_1} \left\| \mathbf{V}_{i_1}^{i}-\bar{\mathbf{V}}^i \right\| ^2=0$, we just set $\mathbf{S}_{i_1}^{corr,i}=\mathbf{0}$.
The additional term $\mathbf{S}^{corr}$ is a correction term to ensure entropy dissipation. Similar techniques are used in \cite{abgrall2018general, chen2020review, gaburro2023high, liu2024non}. In contrast to the aforementioned works, it can be seen from \eqref{eq:scheme1D} that this term is entirely local, depending solely on information from $K_i$. Moreover, it is a high-order term that does not affect mass conservation or accuracy. These properties will be analyzed in the next section.
Compared to the WB scheme in \cite{xu2024high}, our method employs entropy conservative fluxes and incorporates an additional entropy correction term to prevent total entropy increase. In addition, unlike the ES scheme in \cite{liu2025structure}, which is WB only for hydrostatic equilibria, our scheme achieves the WB property for general equilibrium states by carefully modifying the discretization of the source term.

One can rewrite the semi-discrete nodal DG formulation \eqref{eq:scheme1D} as a system of ODEs
$$
\frac{\mathrm{d} \mathbf{U}_h}{\mathrm{d}t} = \mathcal{L}_h(\mathbf{U}_h).
$$
Then, the fully-discrete scheme can be obtained by applying an ODE solver. A strong stability preserving Runge-Kutta (SSP-RK) technique is adopted, consisting of convex combinations of explicit Euler steps. To further ensure the PP property, the well-established PP limiter \cite{zhang2010positivity} is incorporated here on each stage.
The discussion that follows centers on the fully-discrete scheme with forward Euler time discretization, represented by:
\begin{equation}\label{eq:EF0}
\mathbf{U}_h^{EF} = \mathbf{U}_h^n + \Delta t \cdot \mathcal{L}_h(\mathbf{U}_h^n).
\end{equation}
To preserve positivity of density and pressure variables at each Gauss-Lobatto integration point, the high-order scaling-based PP limiter $\Pi_h$ is implemented:
\begin{equation}\label{eq:EF}
\mathbf{U}^{n+1}_h = \Pi_h \,\mathbf{U}_h^{EF}= \Pi_h \left( \mathbf{U}_h^n + \Delta t \cdot \mathcal{L}_h(\mathbf{U}_h^n) \right).
\end{equation}
The main idea of the limiter $\Pi_h$ is scaling the numerical solutions towards the cell averages in $K_i$
$$
 \bar{\mathbf{U}}^{i, EF} 
 = \frac{1}{\Delta x} \int_{K_i} \mathbf{U}_h^{EF}(x) \, \mathrm{d}x 
 = \sum_{i_1=0}^k \frac{\omega_{i_1}}{2} \mathbf{U}_{i_1}^{i, EF},  
 $$
using two parameters $\theta_i^{(1)}, \theta_i^{(2)} \in [0, 1]$:
\begin{equation}\label{pp}
\begin{aligned}
\tilde{\rho}^{i,EF}_{i_1} &= \left( 1 - \theta_i^{(1)} \right) \bar{\rho}^{i,EF} + \theta_i^{(1)} \rho^{i,EF}_{i_1}, \\
\left( \begin{array}{c}
    \rho_{i_1} \\
    m_{i_1} \\
    \mathcal{E}_{i_1}
\end{array} \right)^{i,n+1} &= \left( 1 - \theta_{i}^{(2)} \right) \left( \begin{array}{c}
    \bar{\rho} \\
    \bar{m} \\
    \bar{\mathcal{E}}
\end{array} \right)^{i,EF} + \theta_{i}^{(2)} \left( \begin{array}{c}
    \tilde{\rho}_{i_1} \\
    m_{i_1} \\
    \mathcal{E}_{i_1}
\end{array} \right)^{i,EF}.
\end{aligned}
\end{equation}
The values of $\theta$ are given as
\begin{equation} \label{eq:rhoPP}
\theta_i^{(1)} = \min\left\{ \frac{\bar{\rho}^{i,EF} - \varepsilon}{\bar{\rho}^{i,EF} - \rho_m}, 1 \right\}, \quad 
\rho_m = \min\limits_{0 \le i_1 \le k} \rho^{i,EF}_{i_1},
\end{equation}
\begin{equation}\label{eq:tPP}
\theta_i^{(2)} = \min\limits_{0 \le i_1 \le k}\{ t_{i_1} \}, \quad t_{i_1} = \begin{cases}
    1, & p\left( \widetilde{\mathbf{U}}_{i_1}^{i,EF} \right) \ge \varepsilon, \\
    \tilde{t}_{i_1}, & p\left( \widetilde{\mathbf{U}}_{i_1}^{i,EF} \right) < \varepsilon.
\end{cases}
\end{equation}
Here, the parameter $\varepsilon=10^{-13}$ is a small number. 
In \eqref{eq:tPP}, $\widetilde{\mathbf U}_{i_1}^{i,EF}=(\tilde\rho_{i_1}^{i,EF},m_{i_1}^{i,EF},\mathcal E_{i_1}^{i,EF})^T$, and the scaling factor $\tilde t_{i_1}$ is solved from the nonlinear equation
$$ p\left((1-\tilde t_{i_1})\bar{\mathbf U}^{i,n} + \tilde t_{i_1}\widetilde{\mathbf U}_{i_1}^{i,n}\right) = \varepsilon.$$
It is easy to check that, if $\mathbf U_{i_1}^{i,EF}\in \mathscr{G}$ for all $i_1$, then $\theta_i^{(1)}=\theta_i^{(2)}=1$. 
Additionally, it is important to note that the cell average remains unchanged under $\Pi_h$, meaning $\bar{\mathbf U}^{i,n+1} = \bar{\mathbf U}^{i,EF}$. 
When $\bar{\mathbf U}^{i,EF}$ belongs to $\mathscr{G}$, the limiting procedure in equation \eqref{pp} guarantees that the nodal values at all Gauss-Lobatto points satisfy  $\mathbf{U}_{i_1}^{i,n+1} \in \mathscr{G}$. 

\subsection{Properties of the scheme}

\begin{Thm}[Mass-conservation]
The scheme \eqref{eq:scheme1D} conserves the mass. 
\end{Thm}

\begin{proof}
Utilizing the Gauss-Lobatto quadrature on \eqref{eq:scheme1D} on cell $K_i$ yields 
\begin{align*} \Delta x\frac{\mathrm d\bar {\mathbf U}^i}{\mathrm dt}+&\sum\limits_{i_1=0}^k\sum\limits_{l=0}^k2S_{i_1,l}\mathbf F^S(\mathbf U_{i_1}^i,\mathbf U_l^i)+(\hat{\mathbf F}_{i+1/2}-\hat{\mathbf F}_{i-1/2})-(\mathbf F_{k}^i-\mathbf F_0^i) 
\\=&\sum\limits_{i_1=0}^k\sum\limits_{l=0}^k2S_{i_1,l}\mathbf F^S(\mathbf U_{i_1}^{e,i},\mathbf U_l^{e,i})-\sigma_i \sum\limits_{i_1=0}^k \frac{\omega_{i_1}\Delta x}{2}(\mathbf V_{i_1}^i-\bar{\mathbf V}^i)
\\&+\sum\limits_{i_1=0}^k\frac{\omega_{i_1}\Delta x}{2}\left(\mathbf S_{i_1}^{i}-\mathbf S_{i_1}^{e,i}\right) .
\end{align*}
Notice that
$$
\sum_{i_1=0}^k{\frac{\omega _{i_1}}{2}\left( \mathbf{V}_{i_1}^{i}-\bar{\mathbf{V}}^i \right)}=\sum_{i_1=0}^k{\frac{\omega _{i_1}}{2}\mathbf{V}_{i_1}^{i}}-\left( \sum_{i_1=0}^k{\frac{\omega _{i_1}}{2}} \right) \bar{\mathbf{V}}^i=\bar{\mathbf{V}}^i-\bar{\mathbf{V}}^i=0.
$$
And in \cite{chen2017entropy}, it is proved that
\begin{align*}
\sum_{i_1=0}^k&{\sum_{l=0}^k{2S_{i_1,l}\mathbf{F}^S\left( \mathbf{U}_{i_1}^{i},\mathbf{U}_{l}^{i} \right)}}
=\mathbf{F}^{i}_k-\mathbf{F}^{i}_0.
\end{align*}
Similarly, we have $$\displaystyle\sum\limits_{i_1=0}^k\displaystyle\sum\limits_{l=0}^k 2S_{i_1,l}\mathbf F^S(\mathbf U_{i_1}^{e,i},\mathbf U_{l}^{e,i})=\mathbf F^{e,i+1/2}-\mathbf F^{e,i-1/2}.$$
Here, $\mathbf F^{e,i\pm1/2}=\mathbf F(\mathbf U^{e}(x_{i\pm 1/2}))$. Moreover, the first component of $\mathbf S$ and $\mathbf S^e$ are exactly zero.
Therefore, the cell-average of mass satisfies
$$ \Delta x\frac{\mathrm d\bar\rho^i }{\mathrm dt}+\hat F_{1,i+1/2}-\hat F_{1,i-1/2}= F^{e}_{1,i+1/2}-F^{e}_{1,i-1/2}, $$
where $F_1$ denotes the first component of $\mathbf F$. 
For periodic or compact support boundaries, summing over $i$ yields
$$
\frac{\mathrm{d}}{\mathrm{d}t}\left(\Delta x\sum_{i=1}^N{\bar{\rho}^i}\right)=0.
$$
\end{proof}

\begin{Thm}[Accuracy]
The scheme \eqref{eq:scheme1D} is high-order accurate if the solution $\mathbf U$ and the steady-state solution $\mathbf U^e$ are both sufficiently smooth. In particular, the truncation error is
\begin{equation}\label{eq:truncation1}
    \mathcal L_h(\mathbf U)-\mathcal L(\mathbf U)=\mathcal O(\Delta x^k)+\frac{\mathcal O(\Delta x^{k+1})}{\mathcal O(\Delta x^2)},
\end{equation}
where $\mathcal L(\mathbf U)=-\mathbf F(\mathbf U)_x+\mathbf S(\mathbf U,x)$. 
\end{Thm}

\begin{proof}
Substituting the exact solution $\mathbf U$ and $\mathbf U^e$ into \eqref{eq:scheme1D} yields
\begin{equation}\begin{aligned}
\left( \mathcal{L} _h\left( \mathbf{U} \right) -\mathcal{L} \left( \mathbf{U} \right) \right) |_{x_i\left( X_{i_1} \right)}=&-\frac{2}{\Delta x}\sum_{l=0}^k{2D_{i_1,l}\mathbf{F}^S\left( \mathbf{U}_{i_1}^{i},\mathbf{U}_{l}^{i} \right)}+\mathbf{F}\left( \mathbf{U} \right) _x|_{x_i\left( X_{i_1} \right)}
\\
&-\frac{2}{\Delta x}\frac{\tau _{i_1}}{\omega _{i_1}}\left( \mathbf{F}_{i_1}^{*,i}-\mathbf{F}_{i_1}^{i} \right) 
+\mathbf{S}_{i_1}^{i}-\mathbf{S}\left( \mathbf{U},x_i\left( X_{i_1} \right) \right)
\\
&+\left( \mathbf{S}^0 \right) _{i_1}^{i}-\left( \mathbf{S}^{corr} \right) _{i_1}^{i}. \end{aligned} 
\end{equation}
At first, we want to show that the last two terms $\mathbf S^0,\ \mathbf S^{corr}$ are both high order. For $\mathbf S^0$, similar to \cite{chen2017entropy}, it can be verified that 
$$\begin{aligned}
\frac{2}{\Delta x}\sum_{l=0}^k{2D_{i_1,l}\mathbf{F}^S\left( \mathbf{U}_{i_1}^{e,i},\mathbf{U}_{l}^{e,i} \right)}&=\mathbf F(\mathbf U^e)_x|_{ x_i\left( X_{i_1} \right) } +\mathcal{O} \left( \Delta x^k \right) \\&=\mathbf{S}\left( \mathbf{U}^e,x_i\left( X_{i_1} \right) \right) +\mathcal{O} \left( \Delta x^k \right), \end{aligned}
$$
which indicates $\mathbf S^{0,i}_{i_1} = \mathcal O(\Delta x^k)$ is high-order. For $\mathbf S^{corr}$, since the quadrature is exact for polynomials of degree at most $2k - 1$, we have 
\begin{equation*}\begin{aligned}
\sum_{i_1=0}^k&{\omega _{i_1}\left( \mathbf{V}_{i_1}^{i}-\mathbf{V}_{i_1}^{e,i} \right) ^T\mathbf{S}^{0,i}_{i_1}}
\\&=\sum\limits_{i_1=0}^k\omega_{i_1}(\mathbf V_{i_1}^i-\mathbf V_{i_1}^{e,i})\left(\mathbf F(\mathbf U^e)_x|_{x_i(X_{i_1})}+\mathcal O(\Delta x^k)-\mathbf S(\mathbf U^e,x_i(X_{i_1}))\right)
\\&=\frac{2}{\Delta x}\int_{K_i}{\left( \mathbf{V}-\mathbf{V}^{e} \right) \left( \mathbf F(\mathbf U^e)_x-\mathbf{S}\left( \mathbf{U}^e,x \right) \right) \mathrm{d}x}+\mathcal O(\Delta x^{2k})+\mathcal{O} \left( \Delta x^{k} \right) 
\\&=0+\mathcal{O} \left( \Delta x^{k} \right)=\mathcal O(\Delta x^{k}).
\end{aligned}
\end{equation*}
Further, we have $\mathbf V_{i_1}^{i}-\bar{\mathbf V}^i=\mathcal O(\Delta x)$, which indicates $\mathbf S^{corr,i}_{i_1}=\mathcal O(\Delta x^{k+1})/\mathcal O(\Delta x^2)$. 

Meanwhile, $\mathbf F_{i_1}^{*,i}=\mathbf F_{i_1}^i=\mathbf F(\mathbf U(x_i(X_{i_1})), i_1=0,\ldots,k$ and $\mathbf S_{i_1}^i=\mathbf S(\mathbf U,x_i(X_{i_1}))$ are exact.  
Finally, for the left-hand side terms in \eqref{eq:scheme1D}, it is proved in \cite{chen2017entropy} that
$$
\sum_{l=0}^k{\frac{4}{\Delta x}}D_{i_1,l}\mathbf{F}^S\left( \mathbf{U}_{i_1}^{i},\mathbf{U}_{l}^{i} \right) -\mathbf{F}\left( \mathbf{U} \right) _x|_{x_i(X_{i_1})}=\mathcal{O} \left( \Delta x^k \right) .
$$
By combining all the above, we can obtain \eqref{eq:truncation1}.

\end{proof}
\begin{Remark}
We have proven $\mathbf S^{corr,i}_{i_1}=\mathcal O(\Delta x^{k+1})/\mathcal O(\Delta x^2)$. However, this does not directly imply $\mathbf S^{corr,i}_{i_1}=\mathcal O(\Delta x^{k-1})$, since such an estimate would require the denominator to admit a strictly positive uniform lower bound. Also see \cite{chen2020review}.
\end{Remark}
\begin{Remark}
Although the proven theoretical accuracy is not optimal, the numerical experiments demonstrate the optimal $(k+1)$-th order of accuracy. This observation is consistent with the findings reported in \cite{chen2020review}.
\end{Remark}

\begin{Thm}[Well-balancedness]\label{thm:wb1}
The scheme \eqref{eq:scheme1D} is WB for a general known equilibrium state solution
$\mathbf U^e=(\rho^e,m^e,\mathcal E^e)$, i.e.
\begin{equation}
    \mathcal L_h(\mathbf U_h)=\mathbf 0,\qquad \mathrm{if}\ \ \mathbf U_h=\mathbf U_h^e.
\end{equation}
\end{Thm}
\begin{proof}

Firstly, if $\mathbf U_h=\mathbf U_h^e$, it is straightforward to verify that 
$$ \mathbf S_{i_1}^{i}=\mathbf S_{i_1}^{e,i},\quad \mathbf F^S(\mathbf U_{i_1}^i,\mathbf U_l^i)=\mathbf F^S(\mathbf U_{i_1}^{e,i},\mathbf U_{l}^{e,i}),\quad \mathbf S_{i_1}^{corr,i}=\mathbf 0. $$
Moreover, 
$\mathbf U^e_h$ is continuous at the interfaces $x_{i+1/2}$, since it interpolates the steady state solution $\mathbf U^e$ at Gauss-Lobatto points, which include the interface location. Hence, utilizing the consistency of $\hat{\mathbf F}$, we have 
$\mathbf F_{i_1}^{*,i}=\mathbf F_{i_1}^i,\,i_1=0,k.$ 
Furthermore,
\begin{equation*}\begin{aligned}
\mathcal{L} _h\left( \mathbf{U}_h \right) =& -\frac{2}{\Delta x}\sum_{l=0}^k{2D_{i_1,l}\mathbf{F}^S\left( \mathbf{U}_{i_1}^{i},\mathbf{U}_{l}^{i} \right)}-\frac{2}{\Delta x}\frac{\tau _{i_1}}{\omega _{i_1}}\left( \mathbf{F}_{i_1}^{*,i}-\mathbf{F}_{i_1}^{i} \right) +\mathbf{S}_{i_1}^{i}-\mathbf{S}_{i_1}^{e,i}
\\
&+\frac{2}{\Delta x}\sum_{l=0}^k{2D_{i_1,l}\mathbf{F}^S\left( \mathbf{U}_{i_1}^{e,i},\mathbf{U}_{l}^{e,i} \right)}-(\mathbf{S}^{corr})_{i_1}^{i}
\\
=&\ \mathbf 0.\end{aligned}
\end{equation*}
\end{proof}

\begin{Remark}
For the fully-discrete scheme \eqref{eq:EF}, if $\mathbf U_h^n=\mathbf U_h^e$, then $\mathcal L_h(\mathbf U_h^n)=\mathbf 0$, yielding $\mathbf U_h^{EF}=\mathbf U_h^e$. 
When $\mathbf U^e(x)\in\mathscr G$, we have 
$\mathbf U_{i_1}^{EF,i} =\mathbf U^e(x_i(X_{i_1})) \in\mathscr G$
for all $i,i_1$, resulting in $\theta_i^{(1)}=\theta_i^{(2)}=1$. Therefore, we have $\mathbf U_h^{n+1}=\Pi_h(\mathbf U_h^{EF})=\mathbf U_h^e$, i.e., the fully-discrete scheme \eqref{eq:EF} is also WB.
\end{Remark}

\begin{Remark}
Notably, the design of the scheme and the proof of Theorem \ref{thm:wb1} do not depend on any particular properties of isothermal or isentropic equilibria. As a result, the scheme is, in principle, WB for arbitrary equilibrium states.
\end{Remark}

\begin{Remark}
At boundaries, to achieve the WB property, the boundary treatments should be consistent with the structure of the equilibrium state $\mathbf U^e$. For example, at the left boundary $x=a$, we need $\mathbf U_h(a^-)=\mathbf U_h(a^+)$ when $\mathbf U_h=\mathbf U_h^e$. This condition can be easily verified for outflow and Dirichlet boundaries. For periodic boundaries, if $\mathbf U^e$ is periodic, then  $\mathbf U_h(a^-)=\mathbf U_h(b^-)=\mathbf U^e(b)=\mathbf U^e(a)=\mathbf U_h(a^+)$.  For solid-wall boundaries, if the velocity is continuous, then at $x=a$ we have $u^e=0$, while $\rho^e$ and $\mathcal E^e$ are continuous at $x=a$. Therefore, the standard reflective boundary treatment can also satisfy $\mathbf U_h(a^-)=\mathbf U_h(a^+)$ when $\mathbf U_h=\mathbf U_h^e$.  For more complex boundaries, the situation requires case-by-case analysis. 
\end{Remark}

\begin{Thm}[Entropy stability] 
Assume the boundary is periodic or compact support. The scheme \eqref{eq:scheme1D} is entropy conservative in a single element in the sense of
\begin{equation}
    \frac{\mathrm{d}}{\mathrm{d}t} \left( \sum_{i_1=0}^k{\frac{\omega _{i_1}\Delta x}{2}\mathcal{U} \left( \mathbf{U}_{i_1}^{i} \right)} \right) =-(\mathcal{F} _{k}^{*,i}-\mathcal{F} _{0}^{*,i}),
\end{equation}
and entropy stable in $\Omega$ in the sense of
\begin{equation}
    \frac{\mathrm{d}}{\mathrm{d}t} \left( \sum_{i=1}^N{\sum_{i_1=0}^k{\frac{\omega_{i_1}\Delta x}{2}\mathcal{U} \left( \mathbf{U}_{i_1}^{i} \right)}} \right) \le 0,
\end{equation}
where
$$
\mathcal{F} _{k}^{*,i}=\left( \mathbf{V}_{k}^{i} \right) ^T\hat{\mathbf{F}}_{i+1/2}-\psi _{k}^{i},\quad \mathcal{F} _{0}^{*,i}=\left( \mathbf{V}_{0}^{i} \right) ^T\hat{\mathbf{F}}_{i-1/2}-\psi _{0}^{i}.
$$
\end{Thm}
\begin{proof}

Left multiplying $\mathbf V^T$ in \eqref{eq:scheme1D} and integral in cell $K_i$ yields
\begin{equation*}\begin{aligned}
\frac{\mathrm{d}}{\mathrm{d}t}&\sum_{i_1=0}^k\frac{\omega _{i_1}\Delta x}{2}\mathcal{U} \left( \mathbf{U}_{i_1}^{i} \right) =\sum_{i_1=0}^k{\frac{\omega _{i_1}\Delta x}{2}\left( \mathbf{V}_{i_1}^{i} \right) ^T\frac{\mathrm{d}\mathbf{U}_{i_1}^{i}}{\mathrm{d}t}}
\\
=&-\sum_{i_1=0}^k{\sum_{l=0}^k{2S_{i_1,l}\left( \mathbf{V}_{i_1}^{i} \right) ^T\mathbf{F}^S\left( \mathbf{U}_{i_1}^{i},\mathbf{U}_{l}^{i} \right)}}-\sum_{i_1=0}^k{\tau_{i_1}\left( \mathbf{V}_{i_1}^{i} \right) ^T\left( \mathbf{F}_{i_1}^{*,i}-\mathbf{F}_{i_1}^{i} \right)}
\\
&+\sum_{i_1=0}^k{\frac{\omega _{i_1}\Delta x}{2}\left( \mathbf{V}_{i_1}^{i} \right) ^T}\mathbf{S}_{i_1}^{i}
\\&+\sum_{i_1=0}^k{\frac{\omega _{i_1}\Delta x}{2}\left( \mathbf{V}_{i_1}^{i} \right) ^T\mathbf{S}_{i_1}^{0,i}}-\sigma _i\frac{\Delta x}{2}\sum_{i_1=0}^k{\omega _{i_1}\left(\mathbf V_{i_1}^i\right)^T\left( \mathbf{V}_{i_1}^{i}-\bar{\mathbf{V}}^i \right)}.
\end{aligned}
\end{equation*}
In \cite{chen2017entropy}, it is proved that
\begin{equation}\label{eq:T1}
-\sum_{i_1=0}^k{\sum_{l=0}^k{2S_{i_1,l}\left( \mathbf{V}_{i_1}^{i} \right) ^T\mathbf{F}^S\left( \mathbf{U}_{i_1}^{i},\mathbf{U}_{l}^{i} \right)}}-\sum_{i_1=0}^k{\tau_{i_1}\left( \mathbf{V}_{i_1}^{i} \right) ^T\left( \mathbf{F}_{i_1}^{*,i}-\mathbf{F}_{i_1}^{i} \right)}=-\left( \mathcal{F} _{k}^{*,i}-\mathcal{F} _{0}^{*,i} \right). 
\end{equation}
Now we analysis the entropy produce of source terms. For $\mathbf S_{i_1}^i$, we exactly have
$$
\left( \mathbf{V}_{i_1}^{i} \right) ^T\mathbf{S}_{i_1}^{i}=-\frac{m_{i_1}^{i}}{p_{i_1}^{i}}\rho _{i_1}^{i}\phi _{x,i_1}^{i}+\frac{\rho _{i_1}^{i}}{p_{i_1}^{i}}m_{i_1}^{i}\phi _{x,i_1}^{i}=0,
$$
and similarly, $\left( \mathbf{V}_{i_1}^{e,i} \right) ^T\mathbf{S}_{i_1}^{e,i}=0$. For $\mathbf S^{0,i}_{i_1}$, we can see the extra entropy is balanced with the correction term:
\begin{equation*}\begin{aligned}
\sum_{i_1=0}^k&{\frac{\omega _{i_1}\Delta x}{2}\left( \mathbf{V}_{i_1}^{i} \right) ^T\mathbf{S}_{i_1}^{0,i}}-\sigma _i\frac{\Delta x}{2}\sum_{i_1=0}^k{\omega _{i_1}\left( \mathbf{V}_{i_1}^{i} \right) ^T\left( \mathbf{V}_{i_1}^{i}-\bar{\mathbf{V}}^i \right)}
\\
&=\sum_{i_1=0}^k{\frac{\omega _{i_1}\Delta x}{2}\left( \mathbf{V}_{i_1}^{i} \right) ^T\mathbf{S}_{i_1}^{0,i}}-\sigma _i\frac{\Delta x}{2}\sum_{i_1=0}^k{\omega _{i_1}\left( \mathbf{V}_{i_1}^{i}-\bar{\mathbf{V}}^i \right) ^T\left( \mathbf{V}_{i_1}^{i}-\bar{\mathbf{V}}^i \right)}
\\
&=\sum_{i_1=0}^k{\frac{\omega _{i_1}\Delta x}{2}\left( \mathbf{V}_{i_1}^{i} \right) ^T\mathbf{S}_{i_1}^{0,i}}-\sum_{i_1=0}^k{\frac{\omega _{i_1}\Delta x}{2}\left( \mathbf{V}_{i_1}^{i}-\mathbf{V}_{i_1}^{e,i} \right) ^T\mathbf{S}_{i_1}^{0,i}}
\\
&=\sum_{i_1=0}^k{\frac{\omega _{i_1}\Delta x}{2}\left( \mathbf{V}_{i_1}^{e,i} \right) ^T\mathbf{S}_{i_1}^{0,i}}=\sum_{i_1=0}^k{\sum_{l=0}^k{2S_{i_1,l}\left( \mathbf{V}_{i_1}^{e,i} \right) ^T\mathbf{F}^S\left( \mathbf{U}_{i_1}^{e,i},\mathbf{U}_{l}^{e,i} \right)}}
\\
&=\mathcal{F}^{e,i+1/2}-\mathcal{F}^{e,i-1/2}.
\end{aligned}\end{equation*}
Here, $\mathcal F^{e,i\pm 1/2}=\mathcal F(\mathbf U^e(x_{i\pm 1/2}))$. Notice that 
\begin{equation}\label{eq:ES1D}
\mathcal F(\mathbf U^e(x))=\frac{-m^e(x)s^e(x)}{\gamma - 1}\equiv\mathrm{Const},
\end{equation}
which leads to
$$
\sum_{i_1=0}^k{\frac{\omega _{i_1}\Delta x}{2}\left( \mathbf{V}_{i_1}^{i} \right) ^T\mathbf{S}_{i_1}^{0,i}}-\sigma _i\frac{\Delta x}{2}\sum_{i_1=0}^k{\omega _{i_1}\left( \mathbf{V}_{i_1}^{i} \right) ^T\left( \mathbf{V}_{i_1}^{i}-\bar{\mathbf{V}}^i \right)}=0.
$$
Hence, we have proved that
\begin{equation}\label{eq:ECDG}
\sum_{i_1=0}^k{\frac{\omega _{i_1}\Delta x}{2}\frac{\mathrm{d}\mathcal{U} \left( \mathbf{U}_{i_1}^{i} \right)}{\mathrm{d}t}}=-(\mathcal{F} _{k}^{*,i}-\mathcal{F} _{0}^{*,i}).
\end{equation}
Moreover, for an entropy stable flux $\hat{\mathbf{F}}$, \cite{chen2017entropy} demonstrated that the following inequality holds at each cell interface:
\begin{equation}\label{eq:FhatES}
-(\mathcal{F} _{k}^{*,i}-\mathcal{F} _{0}^{*,i+1})=\left( \psi _{k}^{i}-\psi _{0}^{i+1} \right) -\left( \mathbf{V}_{k}^{i}-\mathbf{V}_{0}^{i+1} \right) ^T\hat{\mathbf{F}}_{i+1/2}\le 0.
\end{equation}
Summing \eqref{eq:ECDG} over $i$ yields
$$
\sum_{i=1}^N{\sum_{i_1=0}^k{\frac{\omega_{i_1}\Delta x}{2}\frac{\mathrm{d}\mathcal{U} \left( \mathbf{U}_{i_1}^{i} \right)}{\mathrm{d}t}}}\le 0.
$$
    
\end{proof}

\begin{Thm}[Positivity-preserving]\label{thm:pp}
The fully-discrete scheme \eqref{eq:EF0} is positivity-preserving, i.e. 
\begin{equation}
\mathbf U_{i_1}^{i,n}\in\mathscr G\quad\Rightarrow\quad %{\mathbf U}^{i,n+1}_{i_1},\ 
\bar{\mathbf U}^{i,n+1}\in\mathscr G
\end{equation}
under the time step restriction
\begin{equation}\label{eq:CFLPP}
\Delta t <\min \left\{ \frac{\omega _0}{4\alpha _0}\Delta x,\min_{i,i_1} \left\{ t^{s,i}_{i_1} \right\} \right\},\quad \alpha_0=\max\limits_{\Omega}\left\{ \left|u\right|+c \right\}, 
\end{equation}
with $\mathbf S_{i_1}^{0,i} =: \left( \Theta_{1,i_1}^i,\ \Theta_{2,i_1}^i,\ \Theta_{3,i_1}^i \right)^T$ and 
\begin{equation}\label{eq:ts1D}
t_{i_1}^{s,i}=\begin{cases}
\displaystyle \frac{1}{2}\min\left\{\frac{B+\sqrt{B^2+4AC}}{2A},\frac{(1-K)\rho^{i}_{i_1}}{\left|\Theta_{1,i_1}^i\right|}\right\}, \quad \text{if} \quad \Theta_{1,i_1}^i\neq0, \\
\displaystyle \frac{B+\sqrt{B^2+4AC}}{4A}, \hspace{4.5cm} \text{if} \quad \Theta_{1,i_1}^i=0.\\
\end{cases}
\end{equation}
Omit the subscripts $i,i_1$, the parameters at the grid $i,i_1$ are given by
\begin{equation}\begin{aligned}\label{eq:pppara}
A=&\ \frac{1}{2K\rho}(\Theta_{2}-\rho\phi_{x})^2,
\\B=&\  \Theta_3-m\phi_x-\frac{1}{K}(u\Theta_2-m\phi_x),
\\C=&\ \frac{p}{\gamma - 1}-\frac{1-K}{2K\rho}m^2,
\\K=&\ \frac{1+\tilde K}{2},\quad \tilde K=\frac{(\gamma-1)m^2}{(\gamma-1)m^2+\rho p}.
\end{aligned}\end{equation}
\end{Thm}

\begin{proof}
Notice that 
$$\sum\limits_{i_1=0}^k\frac{\omega_{i_1}}{2}\mathbf S^{corr,i}_{i_1} = 0.$$
Hence, the cell average satisfies
$$
\bar{\mathbf{U}}^{i,n+1}-\bar{\mathbf{U}}^{i,n}=-\frac{\Delta t}{\Delta x}\left(\hat{\mathbf{F}}_{i+1/2}-\hat{\mathbf{F}}_{i-1/2}\right)+\sum_{i_1=0}^k{\frac{\omega _{i_1}}{2}\Delta t\left(\mathbf{S}_{i_1}^{i}+\mathbf S_{i_1}^{0,i}\right)}.
$$
Following the ideal in \cite{zhang2011positivity}, we can rewrite it as the convex combination of following terms:
\begin{equation}\label{eq:comb}\begin{aligned}
\bar{\mathbf{U}}^{i,n+1}&=\frac{\omega _0}{4} \left\{\mathbf{U}_{0}^{i,n}-\frac{4}{\omega_0}\lambda \left[ \hat{\mathbf{F}}\left( \mathbf{U}_{0}^{i,n},\mathbf{U}_{k}^{i,n} \right) -\hat{\mathbf{F}}\left( \mathbf{U}_{k}^{i-1,n},\mathbf{U}_{0}^{i,n} \right) \right] \right\} 
\\
&\quad +\frac{\omega _k}{4}\left\{ \mathbf{U}_{k}^{i,n}- \frac{4}{\omega_k} \lambda \left[ \hat{\mathbf{F}}\left( \mathbf{U}_{k}^{i,n},\mathbf{U}_{0}^{i+1,n} \right) -\hat{\mathbf{F}}\left( \mathbf{U}_{0}^{i,n},\mathbf{U}_{k}^{i,n} \right) \right] \right\} \\
&\quad +\sum_{i_1=1}^{k-1}{\frac{\omega_{i_1}}{4} \mathbf{U}_{i_1}^{i,n}}
+ \sum_{i_1=0}^k{\frac{\omega _{i_1}}{4} \left(\mathbf U_{i_1}^i+2\Delta t\left(\mathbf{S}_{i_1}^{i}+\mathbf S_{i_1}^{0,i}\right)\right)}.
\end{aligned}
\end{equation}
According to the convexity of $\mathscr G$, if all these terms belong to $\mathscr G$, we can get $\bar{\mathbf U}^{i,n+1}\in\mathscr G$. It has been proved in \cite{zhang2010positivity} that the first term and second term belong to $\mathscr G$ when
$$ \Delta t<\frac{\omega_0}{4\alpha_0}\Delta x. $$
Therefore, the main difficulty is to show that
\begin{equation}\label{eq:ppsource}\mathbf U_{i_1}^i+2\Delta t\cdot(\mathbf S_{i_1}^i+\mathbf S_{i_1}^{0,i})\in\mathscr G,\quad \forall i,i_1\end{equation}
under the CFL condition \eqref{eq:CFLPP}. Denote $T=2\Delta t$, omit the subscripts $i,i_1$, \eqref{eq:ppsource} yields
\begin{equation}\label{eq:ppeq0}
\rho+\Theta_1T>0,
\end{equation}
\begin{equation}\label{eq:ppeq1}
\mathcal{E} +\left( \Theta _3-m\phi _x \right) T-\frac{\left( m+\left( \Theta _2-\rho \phi _x \right) T \right) ^2}{2(\rho+\Theta_1T)}>0.
\end{equation}
According to the second condition of \eqref{eq:CFLPP}, we have $\rho+\Theta_1\Delta t>K\rho>0$, i.e. \eqref{eq:ppeq0} holds. Moreover, from \eqref{eq:pppara} we can see $\tilde K\in[0,1)$, hence $K\in [1/2,1)$. Therefore, a sufficient condition of \eqref{eq:ppeq1} is
\begin{equation}\label{eq:ppeq2}
\mathcal{E} +\left( \Theta _3-m\phi _x \right) T-\frac{\left( m+\left( \Theta _2-\rho \phi _x \right) T \right) ^2}{2K\rho}>0.
\end{equation}
The above inequality is equivalent to
\begin{equation}
\label{eq:ppeq3}
-AT^2+BT+C>0. 
\end{equation}
It is clear that $A>0$. On the other hand, if $m=0$, then it is obvious that $C>0$. If $m\ne 0$, notice that the definition of $\tilde K$ ensures 
$$ \frac{1}{\tilde K}-1=\frac{p\rho }{(\gamma - 1)m^2 }. $$
In particular, $1>K> \tilde K>0$, so we have 
$$ \frac{1-K}{K\rho}m^2=\left(\frac{1}{K}-1\right)\frac{m^2}{\rho}<\left(\frac{1}{\tilde K}-1\right)\frac{m^2}{\rho}=\frac{p}{\gamma - 1}, $$
which yields
$$C= \frac{p}{\gamma-1}-\frac{1-K}{2K\rho}m^2> \frac{p}{2(\gamma-1)}>0. $$
Hence, the positive solution of \eqref{eq:ppeq3} is 
$$ T<\frac{B+\sqrt{B^2+4AC}}{2A}. $$
Thus, \eqref{eq:ppeq1} also holds under the CFL condition \eqref{eq:CFLPP}, then we have \eqref{eq:ppsource}. 
\end{proof}

\begin{Remark}
The theorem ensures that the cell average computed by \eqref{eq:EF0} belongs to $\mathscr G$. Hence, by applying the scaling-based PP limiter \eqref{pp}, we ensure $\mathbf U^{i,n+1}_{i_1}\in\mathscr G$ for all $0\le i_1\le k$.
\end{Remark}

\begin{Remark}
In numerical experiments, we observe that $t_{i_1}^{s,i}$ is an $\mathcal O(1)$ term as $\Delta x\to 0$, and hence the restriction imposed by the source term $\Delta t<\min\limits_{i,i_1}\{t_{i_1}^{s,i}\}$ is much weaker than the flux term $\Delta t<\omega_0\Delta x/4\alpha_0$.
\end{Remark}

\section{Structure preserving nodal DG method in two dimensions}
\label{sec4}

In the 2D case, the Euler equation with gravity can be written as
\begin{equation}\label{eq:EulerG-2D}
\left[ \begin{array}{c}
	\rho\\
	m\\
	n\\
	\mathcal{E}\\
\end{array} \right] _t+\left[ \begin{array}{c}
	m\\
	\rho u^2+p\\
	\rho uv\\
	u\left( \mathcal{E} +p \right)\\
\end{array} \right] _x+\left[ \begin{array}{c}
	n\\
	\rho uv\\
	\rho v^2+p\\
	v\left( \mathcal{E} +p \right)\\
\end{array} \right] _y=\left[ \begin{array}{c}
	0\\
	-\rho \phi _x\\
	-\rho \phi _y\\
	-m\phi _x-n\phi _y\\
\end{array} \right],
\end{equation}
with $m=\rho u,\ n=\rho v$. The equation is denoted by
\begin{equation}\label{eq:euler2Dcomp} 
\mathbf U_t+\mathbf F(\mathbf U)_x+\mathbf G(\mathbf U)_y=\mathbf S(\mathbf U,x,y). 
\end{equation}
The entropy flux components $\mathcal F_1,\ \mathcal F_2$ referenced in \eqref{eq:entropy} are hereafter written as $\mathcal F,\ \mathcal G$, and the potential flux terms  $\psi_1,\ \psi_2$ from \eqref{eq:psi} are expressed as $\psi_F,\ \psi_G$, respectively. 
Assume that the 2D computational domain $\Omega=[a,b]\times[c,d]$ is divided into an $N_x\times N_y$ uniform rectangular cells $\mathcal K=\{K_{ij}=[x_{i-1/2},x_{i+1/2}]\times [y_{j-1/2},y_{j+1/2}]\}$ with cell center $(x_{i},y_j)$. 
The mesh sizes in $x$- and $y$-direction are represented by $\Delta x$ and $\Delta y$, respectively, where $\Delta x=x_{i+1/2}-x_{i-1/2}$ and $\Delta y =y_{j+1/2}-y_{j-1/2}$.

\subsection{Proposed scheme}

For the 2D case, the equilibrium satisfies
$ \mathbf F(\mathbf U^e)_x+\mathbf G(\mathbf U^e)_y =\mathbf S(\mathbf U^e,x,y).$ 
First, we rewrite the governing equation \eqref{eq:euler2Dcomp} by
$$\mathbf U_t+\mathbf F(\mathbf U)_x+\mathbf G(\mathbf U)_y=\mathbf S(\mathbf U,x,y)+\mathbf F(\mathbf U^e)_x+\mathbf G(\mathbf U^e)_y-\mathbf S(\mathbf U^e,x,y).$$
We introduce the following shorthand notations
$$
x_i\left( X \right) = x_{i} +\frac{\Delta x}{2}X,\quad 
y_j\left( Y \right) = y_{j} +\frac{\Delta y}{2}Y, $$
$$	\mathbf{U}_{i_1,j_1}^{i,j}=\mathbf{U}_h\left( x_i\left( X_{i_1} \right) ,y_j\left( Y_{j_1} \right) \right) ,\quad 
\mathbf{F}_{i_1,j_1}^{i,j}=\mathbf{F}\left( \mathbf{U}_{i_1,j_1}^{i,j} \right) ,\quad 
\mathbf{G}_{i_1,j_1}^{i,j}=\mathbf{G}\left( \mathbf{U}_{i_1,j_1}^{i,j} \right), 
$$
$$
    \mathbf{F}_{i_1,j_1}^{*,i,j}=\left\{ \begin{array}{ll}
    \mathbf{\hat{F}}\left( \mathbf U^{i-1,j}_{k,j_1},\mathbf U^{i,j}_{0,j_1} \right)=:\hat{\mathbf F}^{i-1/2,j}_{j_1} ,& i_1=0,\\	
    \mathbf{0},& 0<i_1<k,\\
    \mathbf{\hat{F}}\left( \mathbf U^{i,j}_{k,j_1},\mathbf U^{i+1,j}_{0,j_1} \right)=:\hat{\mathbf F}^{i+1/2,j}_{j_1} ,& i_1=k,\\	
\end{array}\right.
$$
$$
    \mathbf{G}_{i_1,j_1}^{*,i,j}=\left\{ \begin{array}{ll}
	\mathbf{\hat{G}}\left( \mathbf U^{i,j-1}_{i_1,k},\mathbf U^{i,j}_{i_1,0} \right)=:\hat{\mathbf G}^{i,j-1/2}_{i_1} ,& j_1=0,\\	
        \mathbf{0},& 0<j_1<k,\\
	\mathbf{\hat{G}}\left( \mathbf U^{i,j}_{i_1,k},\mathbf U^{i,j+1}_{i_1,0} \right)=:\hat{\mathbf G}^{i,j+1/2}_{i_1} ,& j_1=k.\\	
	\end{array}\right.
	$$
The proposed 2D nodal DG scheme is given by
\begin{equation}\label{eq:scheme2D}
\begin{aligned}
	\frac{\mathrm{d}\mathbf{U}_{i_1,j_1}^{i,j}}{\mathrm{d}t}=&-\frac{2}{\Delta x}\sum_{l=0}^k{2D_{i_1,l}}\mathbf{F}^S\left( \mathbf{U}_{i_1,j_1}^{i,j},\mathbf{U}_{l,j_1}^{i,j} \right) +\frac{2}{\Delta x}\frac{\tau _{i_1}}{\omega _{i_1}}\left( \mathbf{F}_{i_1,j_1}^{i,j}-\mathbf{F}_{i_1,j_1}^{*,i,j} \right)\\
	&-\frac{2}{\Delta y}\sum_{l=0}^k{2D_{j_1,l}}\mathbf{G}^S\left( \mathbf{U}_{i_1,j_1}^{i,j},\mathbf{U}_{i_1,l}^{i,j} \right) +\frac{2}{\Delta y}\frac{\tau _{j_1}}{\omega _{j_1}}\left( \mathbf{G}_{i_1,j_1}^{i,j}-\mathbf{G}_{i_1,j_1}^{*,i,j} \right)\\
	&+\mathbf{S}_{i_1,j_1}^{i,j}+\left( \mathbf{S}^0 \right) _{i_1,j_1}^{i,j}-\left( \mathbf{S}^{corr} \right) _{i_1,j_1}^{i,j},\\
\end{aligned}
\end{equation}
where
\begin{align*}
&\mathbf{S}_{i_1,j_1}^{i,j}=\left( 0,-\rho _{i_1,j_1}^{i,j}\phi _{x,i_1,j_1}^{i,j},-\rho _{i_1,j_1}^{i,j}\phi _{y,i_1,j_1}^{i,j},-m_{i_1,j_1}^{i,j}\phi _{x,i_1,j_1}^{i,j}-n_{i_1,j_1}^{i,j}\phi _{y,i_1,j_1}^{i,j} \right) ^T ,\\
&\mathbf S_{i_1,j_1}^{e,i,j}=\left( 0,-\rho _{i_1,j_1}^{e,i,j}\phi _{x,i_1,j_1}^{i,j},-\rho _{i_1,j_1}^{e,i,j}\phi _{y,i_1,j_1}^{i,j},-m_{i_1,j_1}^{e,i,j}\phi _{x,i_1,j_1}^{i,j}-n_{i_1,j_1}^{e,i,j}\phi _{y,i_1,j_1}^{i,j} \right) ^T,
\\  
&\mathbf{S}_{i_1,j_1}^{0,i,j}=\frac{2}{\Delta x}\sum_{l=0}^k{2D_{i_1,l}\mathbf{F}^S\left( \mathbf{U}_{i_1,j_1}^{e,i,j},\mathbf{U}_{l,j_1}^{e,i,j} \right)}+\frac{2}{\Delta y}\sum_{l=0}^k{2D_{j_1,l}\mathbf{G}^S\left( \mathbf{U}_{i_1,j_1}^{e,i,j},\mathbf{U}_{i_1,l}^{e,i,j} \right)}-\mathbf S^{e,i,j}_{i_1,j_1},
\\
&\mathbf{S}_{i_1,j_1}^{corr,i,j}=\sigma _{ij}\left( \mathbf{V}_{i_1,j_1}^{i,j}-\bar{\mathbf{V}}^{i,j} \right),\quad \bar{\mathbf{V}}^{i,j}=\sum\limits_{i_1,j_1=0}^{k}\frac{\omega_{i_1}\omega_{j_1}}{4}\mathbf V_{i_1,j_1}^{i,j},
\\
& \sigma _{ij}=\frac{\displaystyle\sum_{i_1,j_1=0}^k{\omega _{i_1}\omega _{j_1}\left( \mathbf{V}_{i_1,j_1}^{i,j}-\mathbf{V}_{i_1,j_1}^{e,i,j} \right) ^T\mathbf{S}_{i_1,j_1}^{0,i,j}}}{\displaystyle\sum_{i_1,j_1=0}^k{\omega _{i_1}\omega _{j_1}\left\| \mathbf{V}_{i_1,j_1}^{i,j}-\bar{\mathbf{V}}^{i,j} \right\| ^2}}.
\end{align*}

Same to the 1D case, coupling with the Euler forward time discretization, the fully-discrete scheme is given by
\begin{equation}\label{eq:EF2D} 
\mathbf U^{n+1}_h = \Pi_h^{2D}\left(\mathbf U_h^n+\Delta t\cdot\mathcal L_h(\mathbf U_h^n)\right).
\end{equation}
Here, $\Pi_h^{2D}$ is the 2D PP limiter, which corrects the density and pressure at each Gauss–Lobatto point to ensure positivity. The definition of $\Pi^{2D}_h$ is similar to $\Pi_h$ in 1D \eqref{pp}. The only difference is that the stencil points on $K_{ij}$ are $\{(X_i(x_{i_1}),Y_j(y_{j_1}))\}_{i_1,j_1=0}^{k}$, which is the tensor product of 1D stencil points along $x$- and $y$-directions.

\subsection{Theoretical properties}

Below, we summarize the theoretical properties of the scheme in 2D. The proofs of the following theorems are analogous to those in the 1D case and are omitted for brevity.

\begin{Thm}[Mass-conservation]
The scheme \eqref{eq:scheme2D} conserves the mass.
\end{Thm}

\begin{Thm}[Accuracy]
The scheme \eqref{eq:scheme2D} is high-order accurate if $\mathbf{U}$ and $\mathbf{U}^e$ are both sufficiently smooth. In particular, the local truncation error is
\begin{equation}\label{eq:truncation2}
    \mathcal L_h(\mathbf U)-\mathcal L(\mathbf U)=\mathcal O(\Delta x^k+\Delta y^k)+\frac{\mathcal O(\Delta x^{k+1}+\Delta y^{k+1})}{\mathcal O(\Delta x^2+\Delta y^2)}. 
\end{equation} 
\end{Thm}

\begin{Thm}[Well-balancedness]
The scheme \eqref{eq:scheme2D} is well-balanced for a general known hydrostatic state solution $\mathbf U^e=(\rho^e,m^e,n^e,\mathcal E^e)$, i.e.
\begin{equation}
\mathcal L_h(\mathbf U_h)=\mathbf 0,\qquad \mathrm{if}\ \ \mathbf U_h=\mathbf U_h^e.
\end{equation}
\end{Thm}

\begin{Thm}[Entropy stability] 
Assume the boundaries are periodic or compact supported. The scheme \eqref{eq:scheme2D} is entropy conservative in a single element in the sense of
\begin{equation} \label{eq:entropy_conservative}
        \begin{aligned}
		&\frac{\mathrm{d}}{\mathrm{d}t}\left( \frac{\Delta x\Delta y}{4}\sum_{i_1,j_1=0}^{k}{\omega _{i_1}\omega _{j_1}\mathcal U(\mathbf U_{i_1,j_1}^{i,j}}) \right) \\&\quad 
        =-\frac{\Delta y}{2}\sum_{j_1=0}^{k}{\omega _{j_1}\left( \mathcal{F} _{k,j_1}^{*,i,j}-\mathcal{F} _{0,j_1}^{*,i,j} \right)}
        +\frac{\Delta y}{2}\sum_{j_1=0}^{k}{\omega _{j_1}\left( \mathcal{F} _{j_1}^{e,i+1/2,j}-\mathcal{F} _{j_1}^{e,i-1/2,j} \right)}
        \\&\quad\quad 
        -\frac{\Delta x}{2}\sum_{i_1=0}^{k}{\omega _{i_1}\left( \mathcal{G} _{i_1,k}^{*,i,j}-\mathcal{G} _{i_1,0}^{*,i,j} \right)}
        +\frac{\Delta x}{2}\sum_{i_1=0}^{k}{\omega _{i_1}\left( \mathcal{G} _{i_1}^{e,i,j+1/2}-\mathcal{G} _{i_1}^{e,i,j-1/2} \right)},
        \end{aligned}
    \end{equation}
    and entropy stable in the sense of
    \begin{equation}\label{eq:entropy_stable}
		\frac{\mathrm d}{\mathrm dt}\mathcal U^{tot}(\mathbf U_h)= \frac{\mathrm{d}}{\mathrm{d}t}\left( \frac{\Delta x\Delta y}{4}\sum_{i,j=1}^{N_x,N_y}{\sum_{i_1,j_1=0}^{k}{\omega _{i_1}\omega _{j_1}\mathcal U(\mathbf U_{i_1,j_1}^{i,j}})} \right) \le 0,
    \end{equation}
    where
    \begin{align*}
		&\mathcal{F} _{k,j_1}^{*,i,j}=\left( \left( \mathbf{V}_{k,j_1}^{i,j} \right) ^T\mathbf{F}_{k,j_1}^{*,i,j}-\psi _{F,k,j_1}^{i,j} \right),\quad 
		\mathcal{F} _{0,j_1}^{*,i,j}=\left( \left( \mathbf{V}_{0,j_1}^{i,j} \right) ^T\mathbf{F}_{0,j_1}^{*,i,j} -\psi _{F,0,j_1}^{i,j}\right),
		\\
		&\mathcal{G} _{i_1,k}^{*,i,j}=\left( \left( \mathbf{V}_{i_1,k}^{i,j} \right) ^T\mathbf{G}_{i_1,k}^{*,i,j} -\psi _{G,i_1,k}^{i,j}\right),\quad \mathcal{G} _{i_1,0}^{*,i,j}=\left(\left( \mathbf{V}_{i_1,0}^{i,j} \right) ^T\mathbf{G}_{i_1,0}^{*,i,j}- \psi _{G,i_1,0}^{i,j} \right),
        \\&\mathcal F_{j_1}^{e,i\pm1/2,j}=\mathcal F(\mathbf U^{e}(x_{i\pm 1/2},y_j(Y_{j_1})),\quad\,\,  \mathcal G_{i_1}^{e,i,j\pm1/2}=\mathcal G(\mathbf U^e(x_i(X_{i_1}),y_{j\pm1/2})).
    \end{align*}
\end{Thm}
\begin{Remark}
A notable distinction from the 1D case arises: the exact equilibrium no longer satisfies $\mathcal F(\mathbf U^e)\equiv \mathrm{Const}$ and $\mathcal G(\mathbf U^e)\equiv\mathrm{Const}$. As a result, the entropy conservative equation differs from \eqref{eq:ECDG}. Nevertheless, the entropy stability property \eqref{eq:entropy_stable} still holds, as the additional terms vanish when summed over $i$ and $j$.
\end{Remark}

\begin{Thm}[Positivity-preserving]
The fully-discrete scheme \eqref{eq:EF2D} is positivity-preserving, i.e. 
\begin{equation}\mathbf U_{i_1,j_1}^{i,j,n}\in\mathscr G\quad\Rightarrow\quad 
\bar{\mathbf U}^{i,j,n+1}\in\mathscr G
\end{equation}
under the time step restriction
\begin{equation}\label{eq:CFLPP2D}
\Delta t <\min \left\{ \frac{\omega _0}{8\alpha _{x,0}}\Delta x,\frac{\omega_0}{8\alpha_{y,0}}\Delta y,\min_{i,j,i_1,j_1} \left\{ t^{s,i,j}_{i_1,j_1} \right\} \right\},
\end{equation}
with
$$\alpha_{x,0}=\max\limits_{\Omega}\left\{ \left|u\right|+c \right\},\quad \alpha_{y,0}=\max\limits_{\Omega}\left\{ \left|v\right|+c \right\}, $$
\begin{equation*}\label{eq:ts2D}
t_{i_1,j_1}^{s,i,j}=\begin{cases}
    \displaystyle \frac{1}{2}\min\left\{\frac{B+\sqrt{B^2+4AC}}{2A},\frac{(1-K)\rho^{i,j}_{i_1,j_1}}{\left|\Theta_{1,i_1,j_1}^{i,j}\right|}\right\}, 
    \quad \text{if}\quad \Theta_{1,i_1,j_1}^{i,j}\neq0,\\
    \displaystyle \frac{B+\sqrt{B^2+4AC}}{4A}, \hspace{4.8cm} \text{if}\quad \Theta_{1,i_1,j_1}^{i,j}=0.
\end{cases}
\end{equation*}
where $\mathbf S_{i_1,j_1}^{0,i,j} =: \left( \Theta_{1,i_1,j_1}^{i,j},\ \Theta_{2,i_1,j_1}^{i,j},\ \Theta_{3,i_1,j_1}^{i,j},\ \Theta_{4,i_1,j_1}^{i,j} \right)^T$. Omit $i,j,i_1,j_1$, the parameters are given by
\begin{equation*}\begin{aligned}
A =&\  
\frac{\left( \Theta _2-\rho \phi _x \right) ^2+\left( \Theta _3-\rho \phi _y \right) ^2}{2K\rho},
\\ B = &\ 
\Theta _4-m\phi_x-n\phi_y-\frac{1}{K}\left( u\Theta _2+v\Theta _3-m\phi _x-n\phi _y \right),
\\ C = &\ 
\frac{p}{\gamma -1}-\frac{1-K}{2K\rho}\left( m^2+n^2 \right), 
\\ K = &\ \frac{1+\tilde K}{2},\quad 
\tilde{K}=\frac{(\gamma - 1)(m^2+n^2)}{\left( \gamma -1 \right) \left( m^2+n^2 \right) +p\rho}.
\end{aligned}
\end{equation*}
\end{Thm}

\begin{Remark}
Same to the 1D case, by using the PP limiter, we can achieve that $\mathbf U_{i_1,j_1}^{i,j,n+1}\in\mathscr G$ for all $0\le i_1,j_1\le k$.
\end{Remark}

\section{Numerical tests}\label{sec5}

This section presents numerical results from our experiments. Results are shown for $k=2$ unless noted otherwise. We employ the 10-stage, fourth-order SSP-RK scheme \cite{ketcheson2008highly} for temporal discretization. The time step is set to
$$ \Delta t =  \frac{\mathrm{CFL}}{\alpha_{x,0}} \Delta x, \quad \text{or} \quad 
\Delta t =  \frac{\mathrm{CFL}}{\alpha_{x,0}/\Delta x+\alpha_{y,0}/\Delta y} ,$$
for 1D or 2D cases respectively,
with $\mathrm{CFL}=0.5$ which satisfies the CFL condition for PP property. To ensure the reliability and robustness of our approach, all numerical experiments are performed without slope limiters. We compare our structure-preserving scheme (termed ``WBESPP") with three alternative methods to demonstrate its superior capability in simultaneously maintaining WB, ES, and PP properties. For one-dimensional problems, these schemes are formulated as:
\begin{itemize}
\item[(a)] ``non-WB" method with the semi-discrete formulation  
\begin{equation*}
\frac{\mathrm{d}\mathbf{U}_{i_1}^{i}}{\mathrm{d}t}+\frac{2}{\Delta x}\sum_{l=0}^k{2D_{i_1,l}\mathbf{F}^S\left( \mathbf{U}_{i_1}^{i},\mathbf{U}_{l}^{i} \right)}+\frac{2}{\Delta x}\frac{\tau _{i_1}}{\omega _{i_1}}\left( \mathbf{F}_{i_1}^{*,i}-\mathbf{F}_{i_1}^{i} \right) =\mathbf{S}_{i_1}^{i}.
\end{equation*}
Note that the scheme is ES due to $(\mathbf V_{i_1}^i)^T\mathbf S_{i_1}^i=0$. 
The limiter $\Pi_h$ is applied on each stage the RK scheme. Hence, the non-WB scheme is ES and PP but not WB.

\item[(b)] ``non-ES" method with the semi-discrete formulation  \begin{equation*}
\frac{\mathrm{d}\mathbf{U}_{i_1}^{i}}{\mathrm{d}t}+\frac{2}{\Delta x}\sum_{l=0}^k{D_{i_1,l}\mathbf{F}_{l}^i}+\frac{2}{\Delta x}\frac{\tau _{i_1}}{\omega _{i_1}}\left( \mathbf{F}_{i_1}^{*,i}-\mathbf{F}_{i_1}^{i} \right) =\mathbf{S}_{i_1}^{i}+(\tilde{\mathbf S}^0)^{i}_{i_1},
\end{equation*}
where 
\begin{equation*}
    \tilde{\mathbf S}_{i_1}^{0,i}=\frac{2}{\Delta x}\sum\limits_{l=0}^kD_{i_1,l}\mathbf F_{l}^i-\mathbf S_{i_1}^{e,i}.
\end{equation*}
The PP limiter $\Pi_h$ is employed. The non-ES scheme is WB and PP but not ES.  

\item[(c)] ``non-PP" method with the same formulation as  \eqref{eq:scheme1D}. 
While the PP limiter $\Pi_h$ is not used. Hence, the non-PP scheme is WB and ES but not PP. 

\end{itemize}

\subsection{One-dimensional tests}

\begin{Ex}
\textbf{(Well-balancedness test.)}
\label{ex:WBmoving}
\end{Ex}

We consider the 1D equilibrium solution with the gravity function $\phi(x)=x$. The computational domain is taken as $\Omega=[0,2]$. We take the same equilibrium state solution in \cite{grosheintz2020well}
\begin{equation}\label{eq:eqbmM}
m^e=-M\gamma^{1/2},\quad H^e+\phi =\frac{\gamma}{\gamma - 1}+\frac{1}{2}M^2\gamma,\quad s^e=0
\end{equation}
with $\gamma = 5/3$. 
The point value at each Gauss-Lobatto point is calculated by solving a nonlinear equation.  
Three equilibrium states in different flow regimes are considered:
\begin{itemize}
\item[(a)] Hydrostatic flow: $M = 0$.
\item[(b)] Subsonic flow: $M = 0.01$.
\item[(c)] Supersonic flow: $M = 2.5$.
\end{itemize}

First, we take the initial value as $\mathbf U_h=\mathbf U^e_h$ and compute the solution until $T=1$. The boundary conditions are set to be the initial values. The errors of $\rho$ are presented in Table \ref{tabwbmoving}. We can see the WB scheme maintains the machine error for numerical equilibrium state. These demonstrate that our scheme is appropriate for WB preservation for both static and moving equilibrium states, while the non-WB scheme fails to preserve the equilibrium state solutions. 

Next, we add a small perturbation to pressure
$$
p=p^e+A \exp \left\{ -100\left( x-x_c \right) ^2 \right\}, 
$$
where $A=10^{-9},\ 10^{-6},\ 10^{-6}$ and $x_c=1,\ 1.1,\ 1.5$ for case (a), (b), (c), respectively. We run the simulation on $N=200$ meshes until $T=0.45$ for (a), (b), and $T=0.25$ for (c). The boundary conditions are set as initial values. The results of pressure perturbation and velocity perturbation are shown in Fig \ref{figWBmoving}. The reference solution is computed by the WB DG method in \cite{xu2024high} on $N=2000$ meshes. Compared to the results in \cite{grosheintz2020well, zhang2024equilibrium, xu2024high}, one can see that the results by WBESPP scheme agree well with the reference solution, while the results by the non-WB scheme do not match the reference ones, indicating the WBESPP method is more accurate for resolving small perturbations to steady states.

\begin{table}[htb!]
        \centering
        \caption{Example \ref{ex:WBmoving}: One-dimensional well-balancedness test. Errors and orders of density at final time $T = 1$ for $k=2$.}
	\setlength{\tabcolsep}{3.2mm}{
		\begin{tabular}{|c|c|cc|cc|cc|}
			\hline 
   &$N$ & $L^1$ error & order & $L^2$ error & order & $L^\infty$ error & order  \\ \hline
   \multicolumn{8}{|c|}{Hydrostatic flow, $M = 0$}\\ \hline	
   \multirow{4}{*}{WBESPP}
   
     &      20 &9.93e-16 & -- &1.34e-15 & -- &3.22e-15 &-- \\
&      40 &1.85e-15 &-- &2.49e-15 & -- &7.55e-15 &-- \\
&      80 &3.66e-15 &--&4.94e-15 & -- &1.38e-14 &-- \\
&     160 &6.71e-15 & -- &8.48e-15 & -- &2.46e-14 &-- \\

			\hline
   \multirow{4}{*}{non-WB}
   
     &      20 &1.17e-06 & -- &1.58e-06 & -- &4.59e-06 &-- \\
&      40 &1.48e-07 & 2.99 &2.00e-07 & 2.98 &5.94e-07 &2.95 \\
&      80 &1.85e-08 & 2.99 &2.52e-08 & 2.99 &7.57e-08 &2.97 \\
&     160 &2.32e-09 & 3.00 &3.16e-09 & 2.99 &9.55e-09 &2.99 \\

			\hline

    \multicolumn{8}{|c|}{Subsonic flow, $M = 0.01$}\\ \hline	
   \multirow{4}{*}{WBESPP}
   
     &      20 &1.18e-15 & -- &1.88e-15 & -- &7.66e-15 &-- \\
&      40 &4.17e-15 & -- &6.48e-15 & -- &2.79e-14 &-- \\
&      80 &1.31e-14 & -- &2.00e-14 & -- &9.81e-14 &-- \\
&     160 &2.77e-14 & -- &4.29e-14 & -- &1.67e-13 &-- \\

			\hline
   \multirow{4}{*}{non-WB}	
    
     &      20 &1.50e-05 & -- &2.89e-05 & -- &1.27e-04 &-- \\
&      40 &3.45e-06 & 2.12 &6.34e-06 & 2.19 &2.79e-05 &2.18 \\
&      80 &7.25e-07 & 2.25 &1.23e-06 & 2.37 &5.31e-06 &2.39 \\
&     160 &1.38e-07 & 2.39 &2.10e-07 & 2.55 &8.37e-07 &2.67 \\

			\hline

    \multicolumn{8}{|c|}{Supersonic flow, $M=2.5$}\\ \hline	
   \multirow{4}{*}{WBESPP}
    
     &      20 &2.82e-14 & -- &4.16e-14 & -- &1.29e-13 &-- \\
&      40 &5.36e-14 & -- &7.67e-14 & -- &2.66e-13 &-- \\
&      80 &1.28e-13 & -- &1.88e-13 & -- &5.65e-13 &-- \\
&     160 &2.69e-13 & -- &3.79e-13 & -- &1.31e-12 &-- \\

			\hline
   \multirow{4}{*}{non-WB}
     
     &      20 &1.85e-05 & -- &2.57e-05 & -- &6.74e-05 &-- \\
&      40 &3.06e-06 & 2.59 &4.21e-06 & 2.61 &1.16e-05 &2.53 \\
&      80 &4.65e-07 & 2.72 &6.36e-07 & 2.73 &1.81e-06 &2.69 \\
&     160 &6.58e-08 & 2.82 &8.96e-08 & 2.83 &2.59e-07 &2.80 \\

			\hline
            
	\end{tabular}} 
    \label{tabwbmoving}
\end{table}

\begin{figure}[htbp!]
	\centering
    \subfigure[Pressure perturbation, $M=0$.]{
		\includegraphics[width=0.45\linewidth]{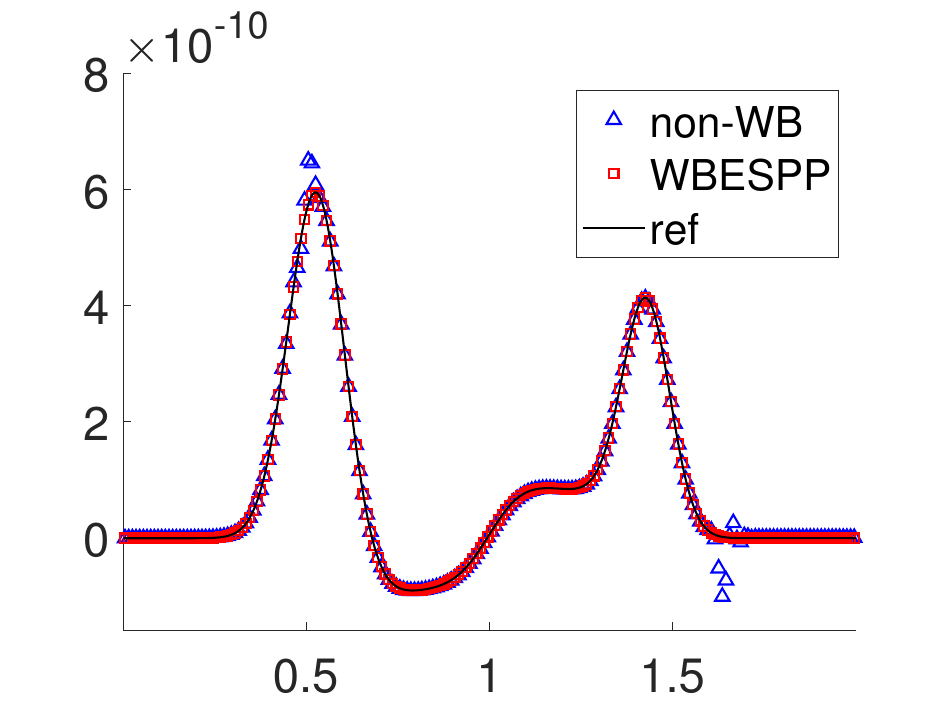}}
	\subfigure[Velocity perturbation, $M=0$.]{
		\includegraphics[width=0.45\linewidth]{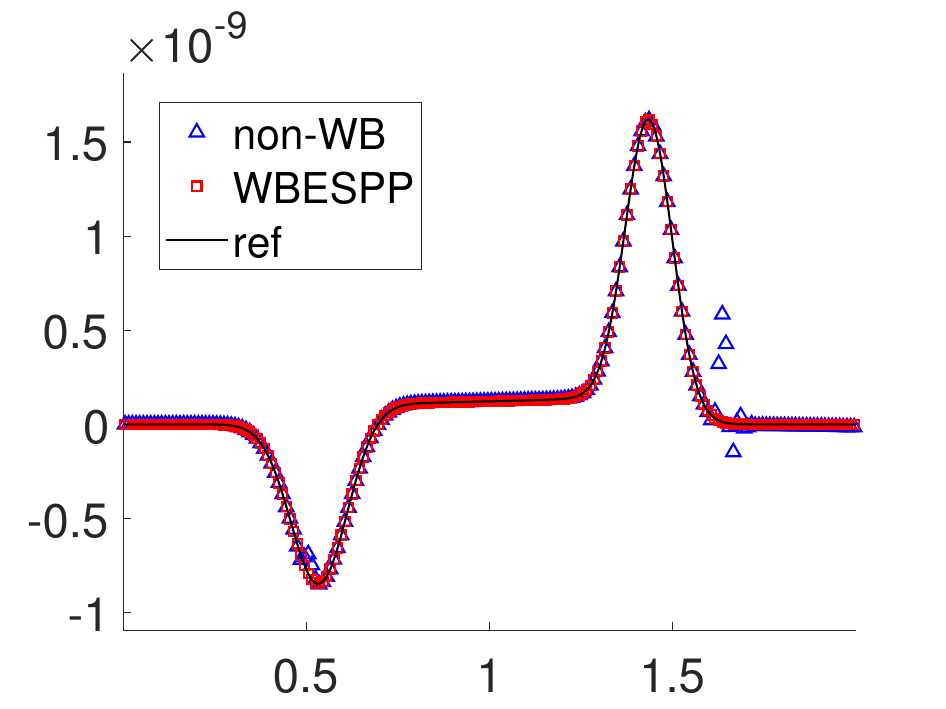}}
    \subfigure[Pressure perturbation, $M=0.01$.]{
		\includegraphics[width=0.45\linewidth]{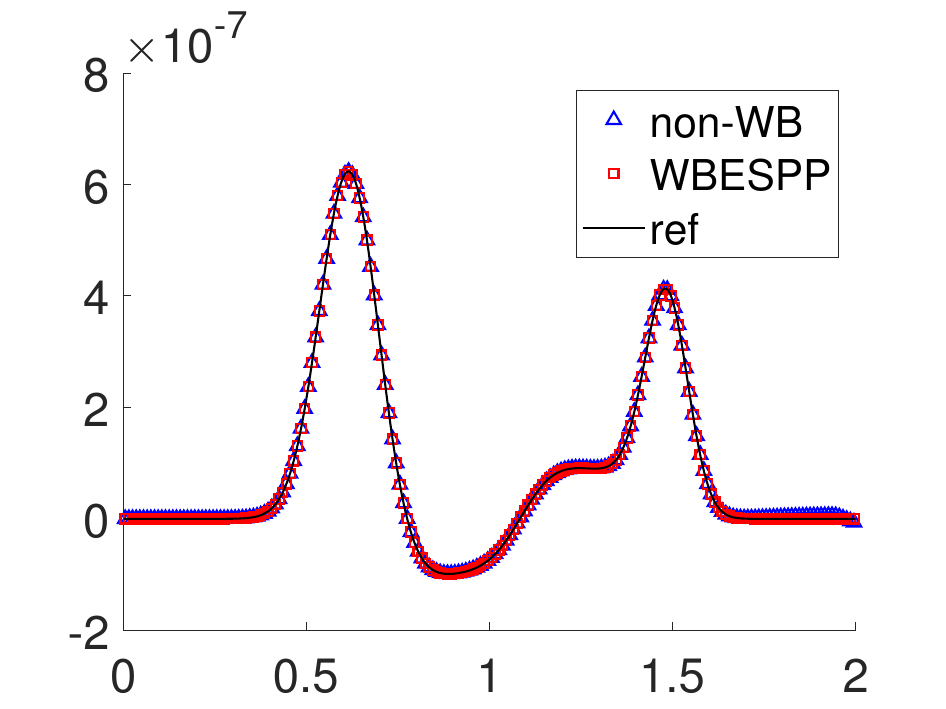}}
	\subfigure[Velocity perturbation, $M=0.01$.]{
		\includegraphics[width=0.45\linewidth]{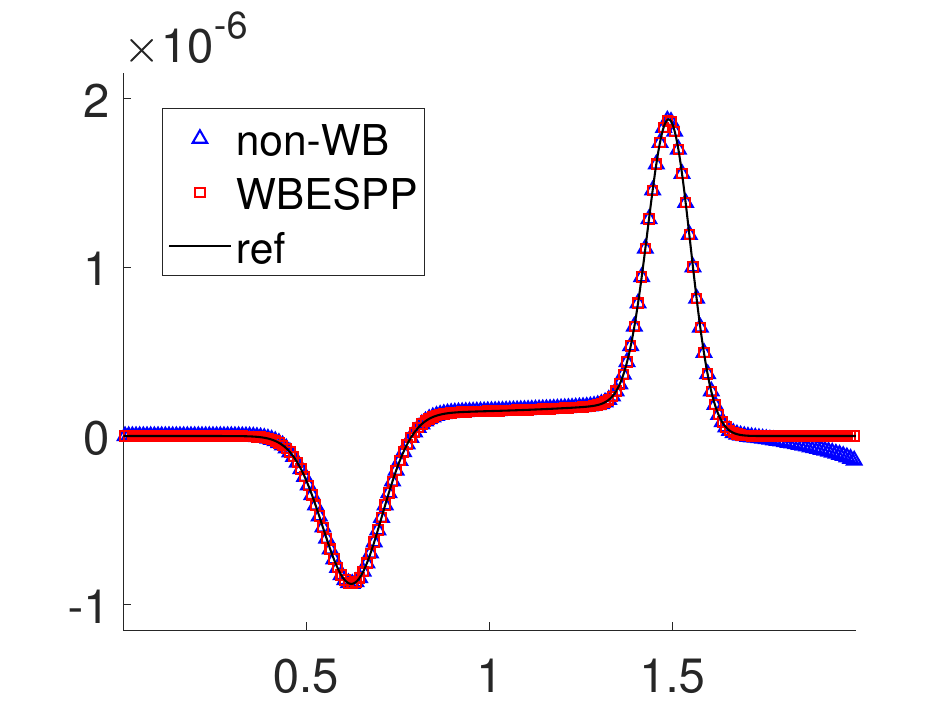}}
    \subfigure[Pressure perturbation, $M=2.5$.]{
		\includegraphics[width=0.45\linewidth]{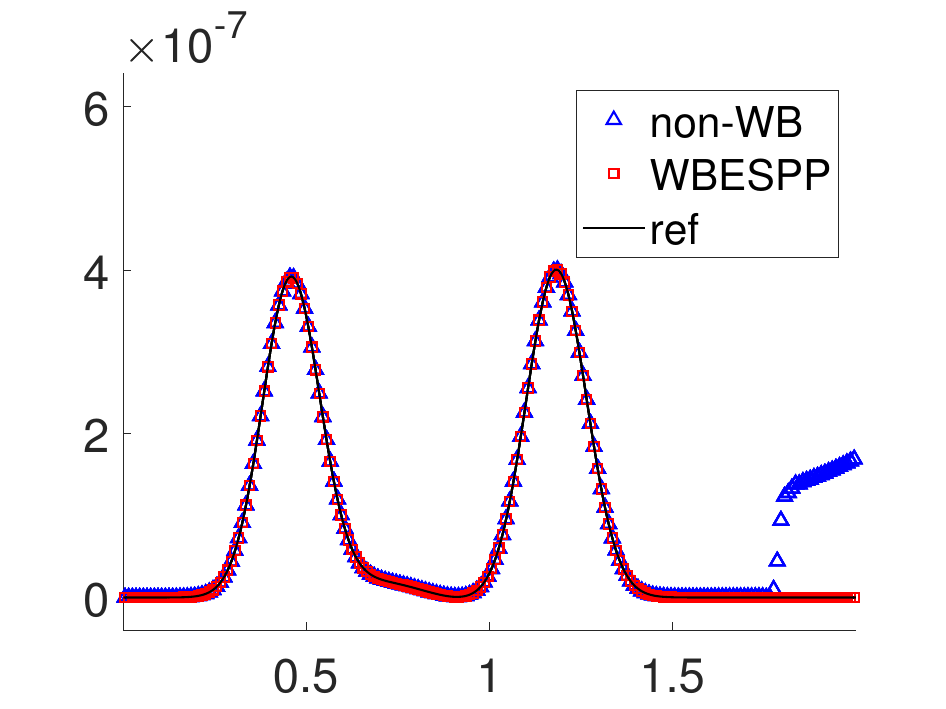}}
	\subfigure[Velocity perturbation, $M=2.5$.]{
		\includegraphics[width=0.45\linewidth]{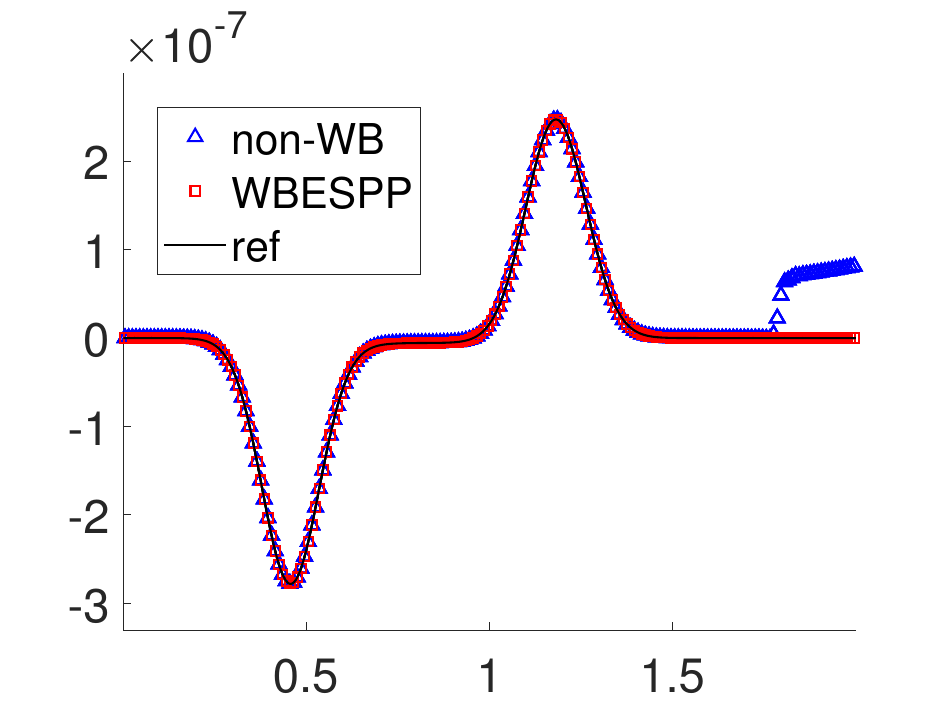}}
	\caption{Example \ref{ex:WBmoving}: One-dimensional well-balancedness test with small pressure perturbation. The numerical solution of pressure perturbation and velocity perturbation of different flow regimes.} 
   \label{figWBmoving}
\end{figure}

\begin{Ex}
\textbf{(Accuracy test.)}
\label{ex:acc1D}
\end{Ex}

This example is considered to test the accuracy of the one-dimensional system. The computational domain is taken as
$\Omega = [0,2]$. The exact smooth solution is given by \cite{xing2013high}
$$
\rho =1+0.2\sin \left( \pi(x-t) \right),\quad u=1,\quad p=5.5-x+t+0.2\cos(\pi(x-t))/\pi.
$$
The boundary conditions are set to be the exact solutions when needed.
In Table \ref{tabacc1D}, we present the errors and orders of $\rho$ at $T=2$ on different meshes. It can be seen that the optimal $(k+1)$-th accuracy is obtained.

\begin{table}[htb!]
        \centering
        \caption{Example \ref{ex:acc1D}: One-dimensional accuracy test. Errors and orders of density at final time $T = 2$ for $k=1,2,3$.}
	\setlength{\tabcolsep}{3.2mm}{
		\begin{tabular}{|c|c|cc|cc|cc|}
			\hline 
   &$N$ & $L^1$ error & order & $L^2$ error & order & $L^\infty$ error & order  \\ \hline	
   \multirow{4}{*}{$k=1$}
   
    &      20 &7.25e-03 & -- &9.65e-03 & -- &2.75e-02 &-- \\
&      40 &1.75e-03 & 2.05 &2.34e-03 & 2.04 &7.65e-03 &1.84 \\
&      80 &4.32e-04 & 2.01 &5.74e-04 & 2.03 &1.99e-03 &1.95 \\
&     160 &1.09e-04 & 1.99 &1.42e-04 & 2.01 &5.01e-04 &1.99 \\

			\hline

    \multirow{4}{*}{$k=2$}
   
     &      20 &2.02e-04 & -- &2.50e-04 & -- &4.63e-04 &-- \\
&      40 &3.07e-05 & 2.72 &3.78e-05 & 2.72 &7.24e-05 &2.68 \\
&      80 &4.14e-06 & 2.89 &5.08e-06 & 2.90 &9.75e-06 &2.89 \\
&     160 &5.28e-07 & 2.97 &6.49e-07 & 2.97 &1.23e-06 &2.98 \\

			\hline

    \multirow{4}{*}{$k=3$}
   
     &      20 &2.18e-06 & -- &3.49e-06 & -- &2.20e-05 &-- \\
&      40 &1.60e-07 & 3.77 &2.42e-07 & 3.85 &1.87e-06 &3.56 \\
&      80 &1.10e-08 & 3.86 &1.58e-08 & 3.93 &1.37e-07 &3.77 \\
&     160 &7.27e-10 & 3.92 &1.02e-09 & 3.96 &9.36e-09 &3.87 \\

			\hline
            
	\end{tabular}} 
    \label{tabacc1D}
\end{table}

\begin{Ex}
\textbf{(Sod shock tube.)}
\label{ex:Sod}
\end{Ex}

We examine a 1D Sod-like shock tube problem featuring a gravitational potential $\phi_x = 1$. This classical benchmark for the Euler equations was first presented in \cite{sod1978survey}. The computational domain $\Omega = [-1,1]$ employs reflective boundary conditions, with initial conditions specified as
$$
\left( \rho ,u, p \right) =\begin{cases}
	\left( 1,0,1 \right) , & x<0,\\
	\left( 0.125,0,0.1 \right) , & x\ge 0.\\
\end{cases}
$$
In Fig \ref{figSod}, we present the result at $T=0.4$ on $N=200$ meshes, where the reference solution is computed by the WBPP DG method in \cite{du2024well} with TVB limiter on $N=2000$ meshes. We note that for the non-ES scheme without limiter, the computation will blow up in the first several steps. 
In Fig \ref{figSodU}, we plot the evolution of total entropy for WBESPP and non-ES schemes. Here we use the equilibrium state solution defined by \eqref{eq:eqbmM} with $M=2.5$ and $\gamma=1.4$. It is notable that since this test problem is far from the equilibrium state, the choice of equilibrium state has no obvious effect on the numerical solutions. 
To illustrate this point, we also test this problem with the hydrostatic state given by $\mathbf U^e=(\exp(-x),\ 0,\ \exp(-x)/(\gamma - 1))$. And in Fig \ref{figSodcompare}, we present the density and the absolute error of these two equilibrium states, denoted by ``moving" and ``hydros", respectively. It can be observed that the absolute error between the two results is very small. Moreover, Fig. \ref{figCFL}(a) shows the maximum time steps restricted by the flux term and the source term in Theorem \ref{thm:pp} for different mesh resolutions, denoted by “F” and “S”, respectively. The results indicate that $t_{i_1}^{s,i}$ in \eqref{eq:ts1D}  remains indeed an $\mathcal O(1)$ term independent of the mesh size, and the time step allowed by the source term is much larger than that allowed by the flux term.

\begin{figure}[htbp!]
	\centering
    \subfigure[Density.]{
		\includegraphics[width=0.31\linewidth]{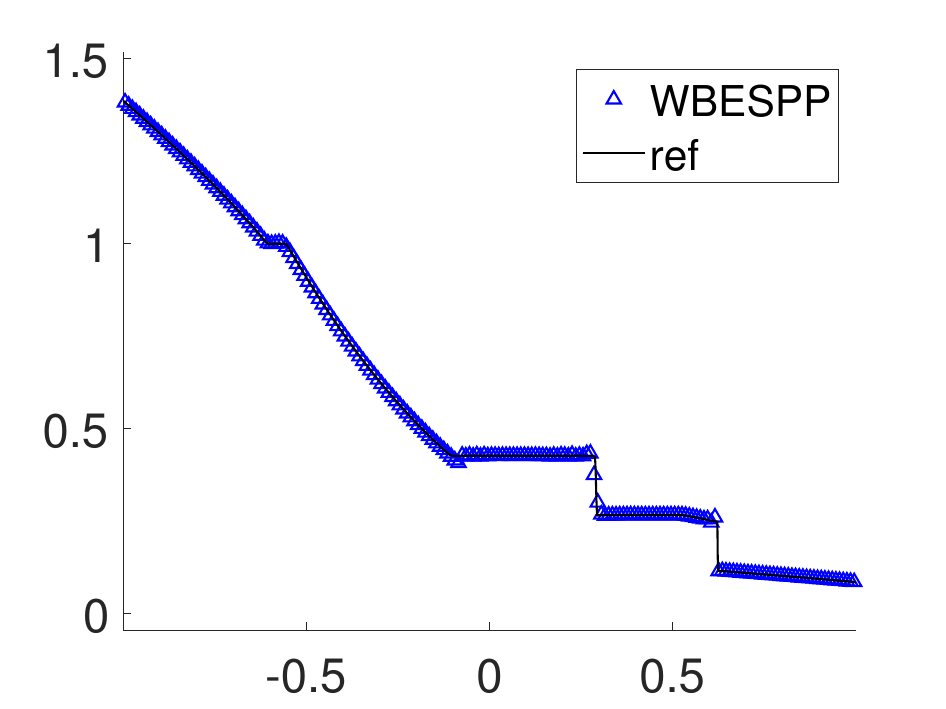}}
	\subfigure[Velocity.]{
		\includegraphics[width=0.31\linewidth]{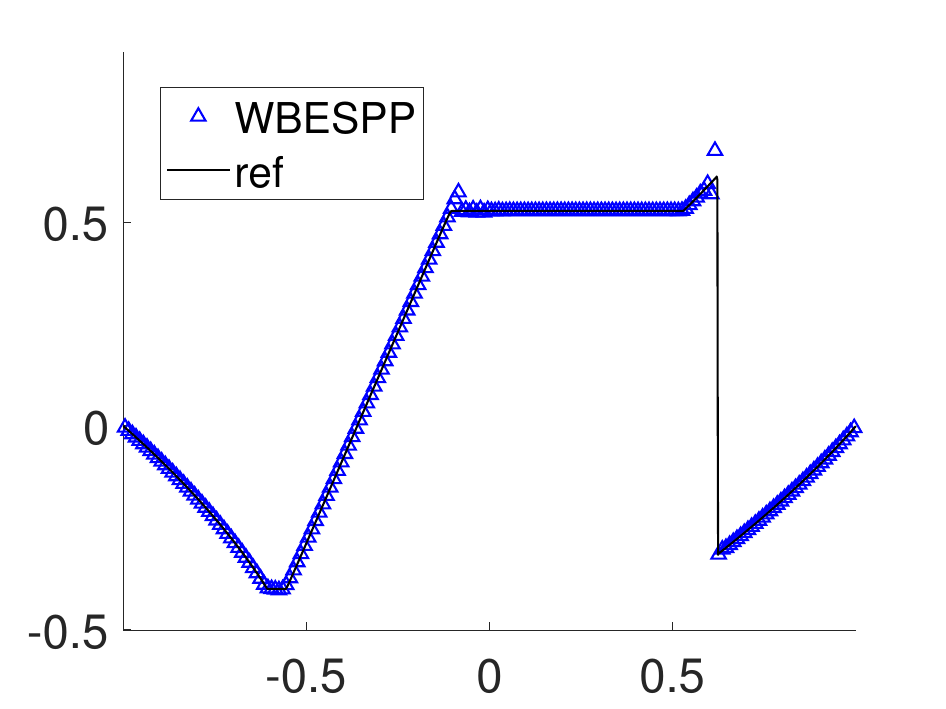}}
    \subfigure[Pressure.]{
		\includegraphics[width=0.31\linewidth]{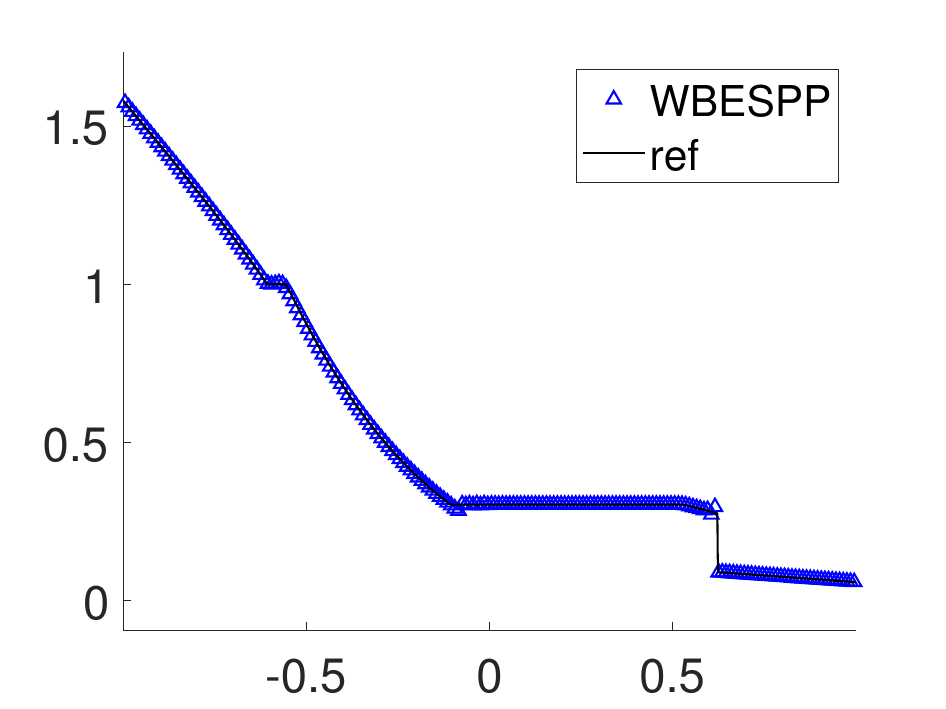}}
	\caption{Example \ref{ex:Sod}: One-dimensional Sod-like shock tube. The numerical solution at $T=0.4$ with $N=200$.}
   \label{figSod}
\end{figure}

\begin{figure}[htbp!]
	\centering
    \subfigure[non-ES.]{   \includegraphics[width=0.45\linewidth]{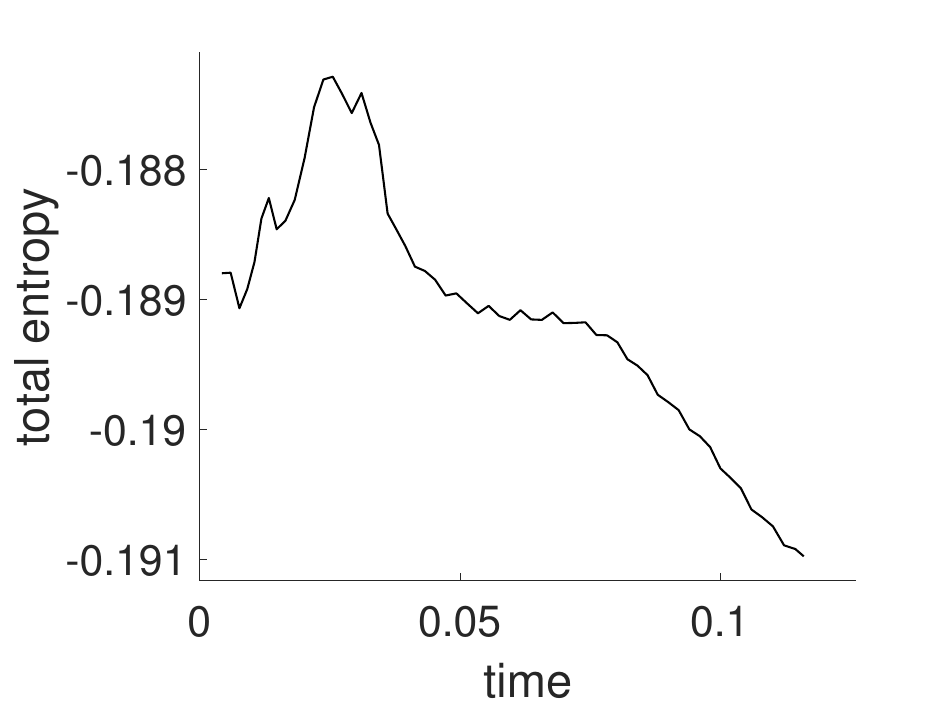}}
	\subfigure[WBESPP.]{ \includegraphics[width=0.45\linewidth]{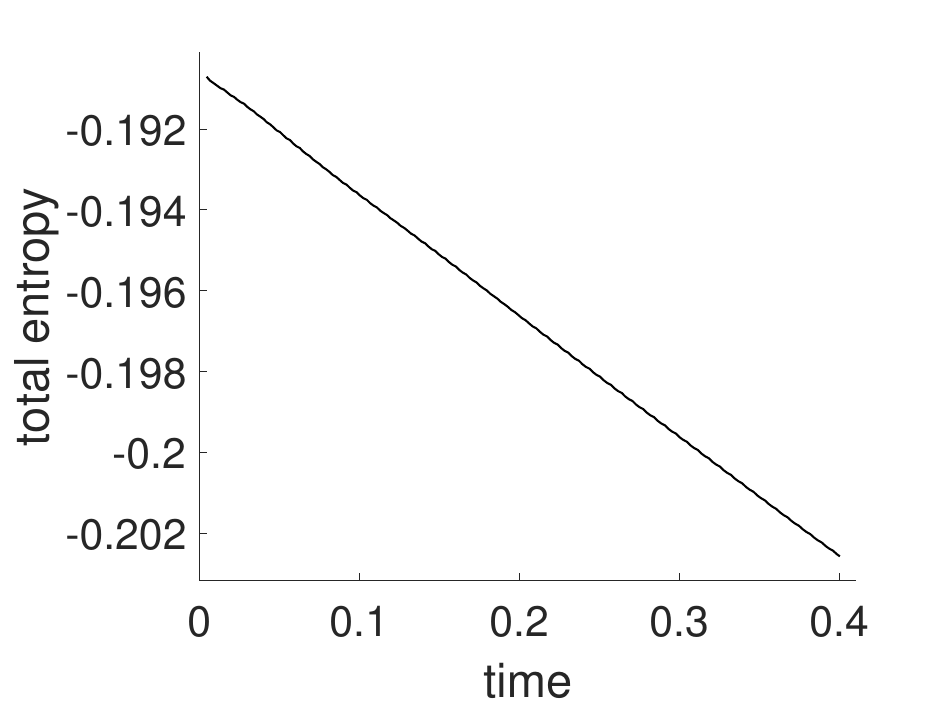}}
	\caption{Example \ref{ex:Sod}: One-dimensional Sod-like shock tube. The evolution of total entropy with time.  
    }
   \label{figSodU}
\end{figure}

\begin{figure}[htbp!]
	\centering
    \subfigure[Density.]{   \includegraphics[width=0.45\linewidth]{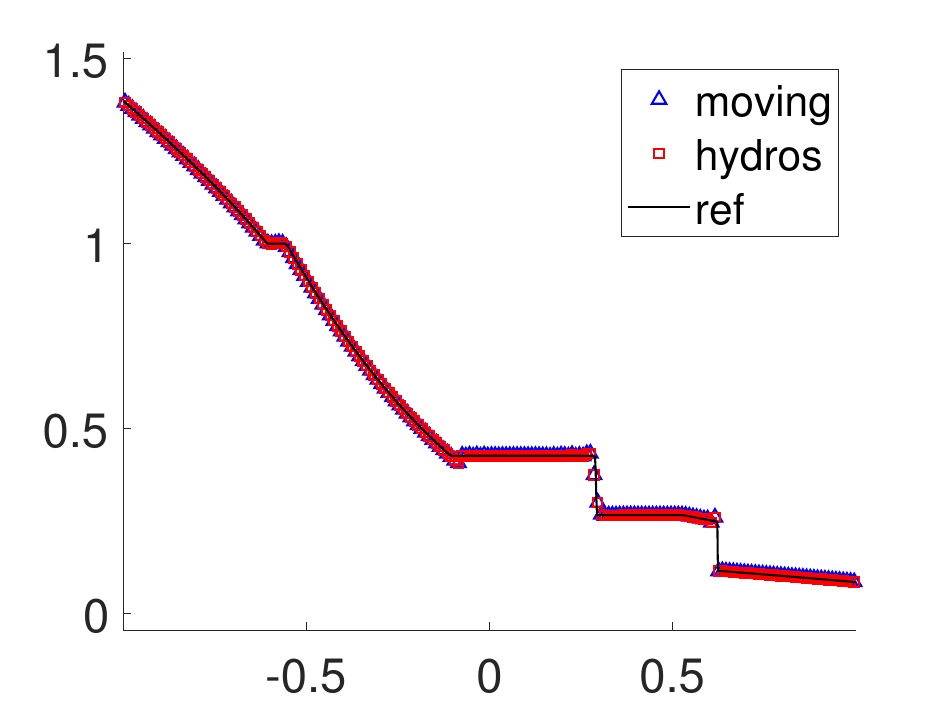}}
	\subfigure[Absolute error.]{ \includegraphics[width=0.45\linewidth]{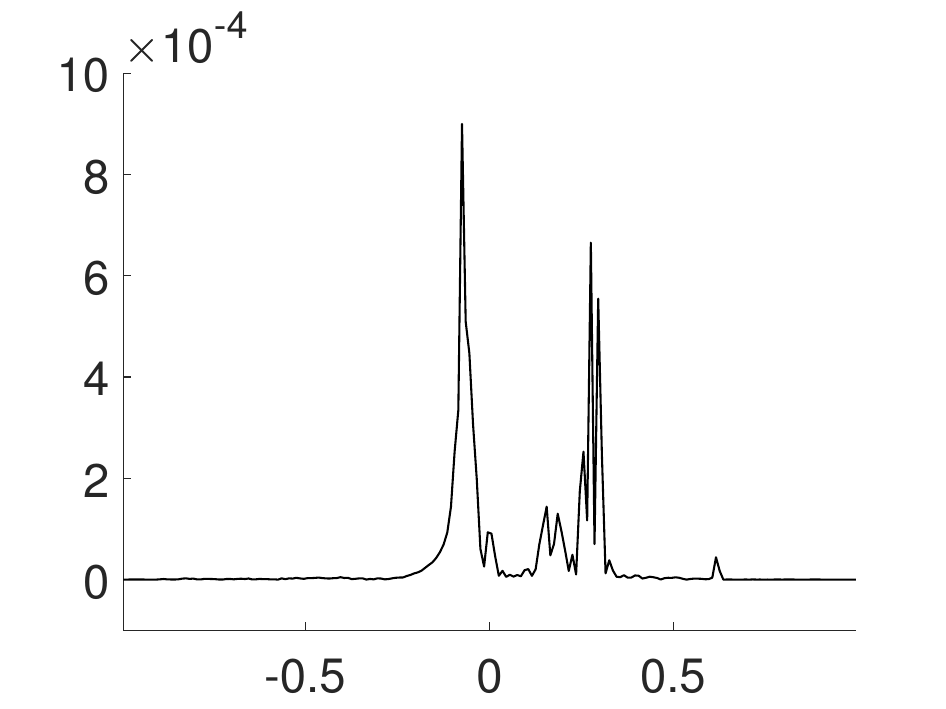}}
	\caption{Example \ref{ex:Sod}: One-dimensional Sod-like shock tube. The density at $T=0.4$ with different equilibrium states.}
   \label{figSodcompare}
\end{figure}

\begin{Ex}
\textbf{(Double rarefaction wave.)}
\label{ex:DRF}
\end{Ex}

We examine the one-dimensional double rarefaction wave problem \cite{wu2021uniformly, du2024well} under the gravitational potential $\phi_x=x$, characterized by extremely low density and pressure regimes. The computational domain $\Omega = [-1,1]$ employs outflow boundary conditions, with initial conditions specified as
$$
\left( \rho ,u,p \right) =\begin{cases}
	\left( 7,-1,0.2 \right) , & x<0,\\
	\left( 7,1,0.2 \right) , & x\ge 0.\\
\end{cases}
$$
In Fig \ref{figDRF}, we present the result at $T=0.6$ on $N=800$ meshes, where the reference solution is computed by the WBPP DG method in \cite{du2024well} on $N=4000$ meshes. For this example, the non-PP scheme will blow up in the first several steps. The maximum time step restricted by the flux term and source term is plotted in Fig \ref{figCFL} (b), which indicates that the restriction imposed by the source term is much weaker.

\begin{figure}[htbp!]
	\centering
    \subfigure[Density.]{
		\includegraphics[width=0.31\linewidth]{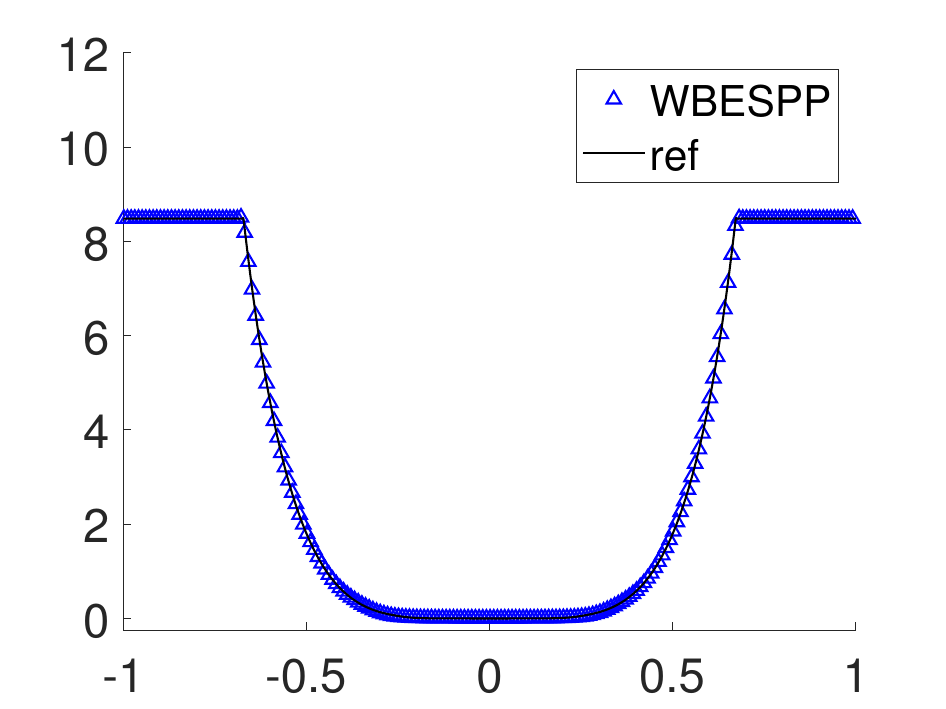}}
	\subfigure[Momentum.]{
		\includegraphics[width=0.31\linewidth]{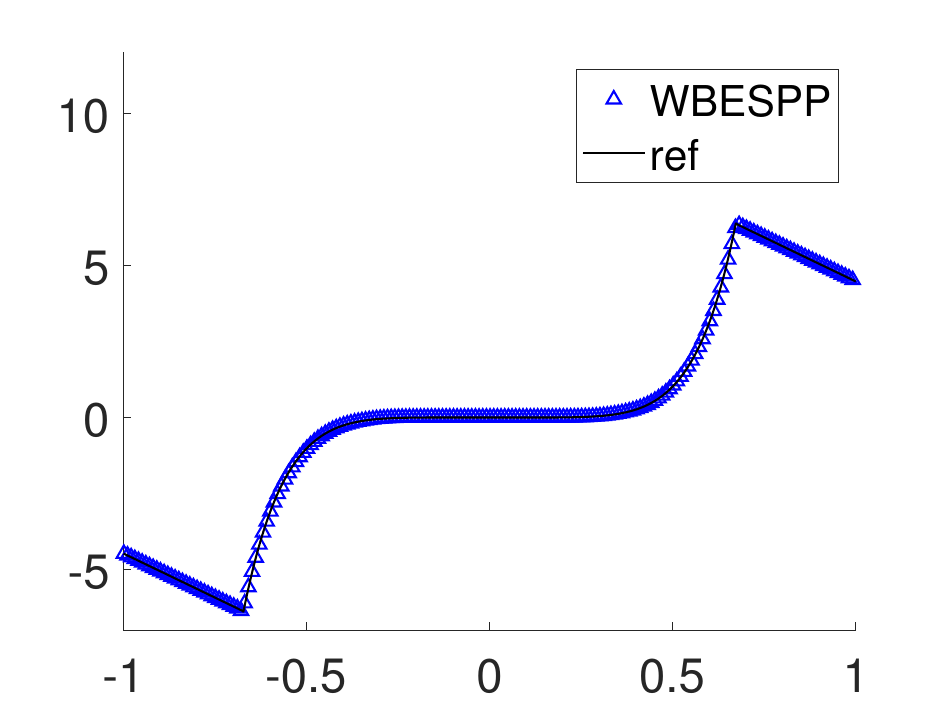}}
	\subfigure[Energy.]{
		\includegraphics[width=0.31\linewidth]{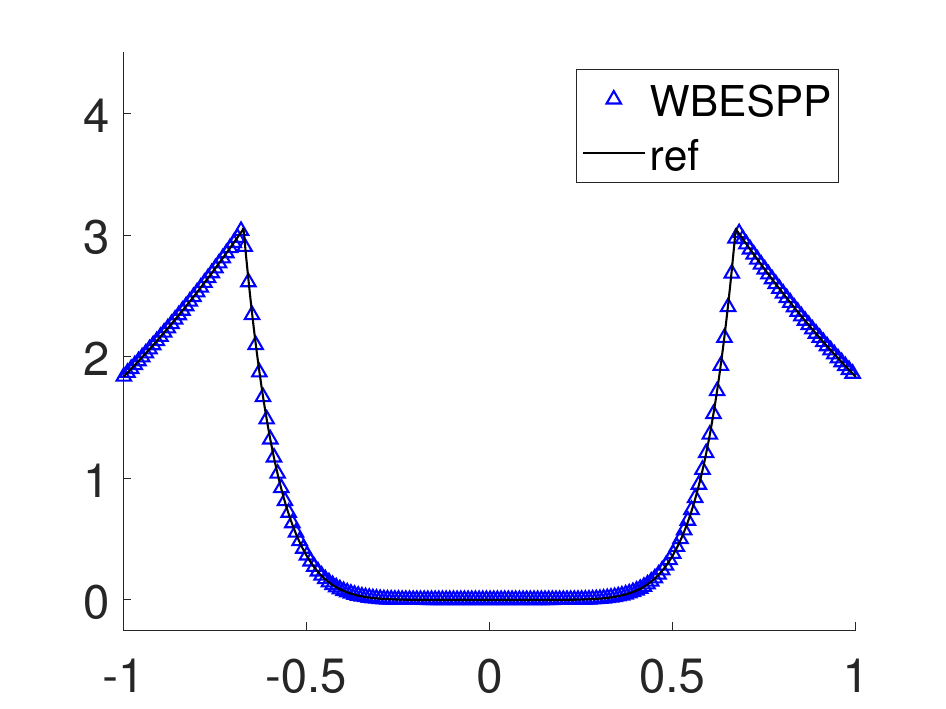}}
	\caption{Example \ref{ex:DRF}: One-dimensional double rarefaction wave problem. The numerical solution at $T=0.6$ with $N=800$.}
   \label{figDRF}
\end{figure}

\begin{figure}[htbp!]
	\centering
    \subfigure[Sod shock tube.]{   \includegraphics[width=0.45\linewidth]{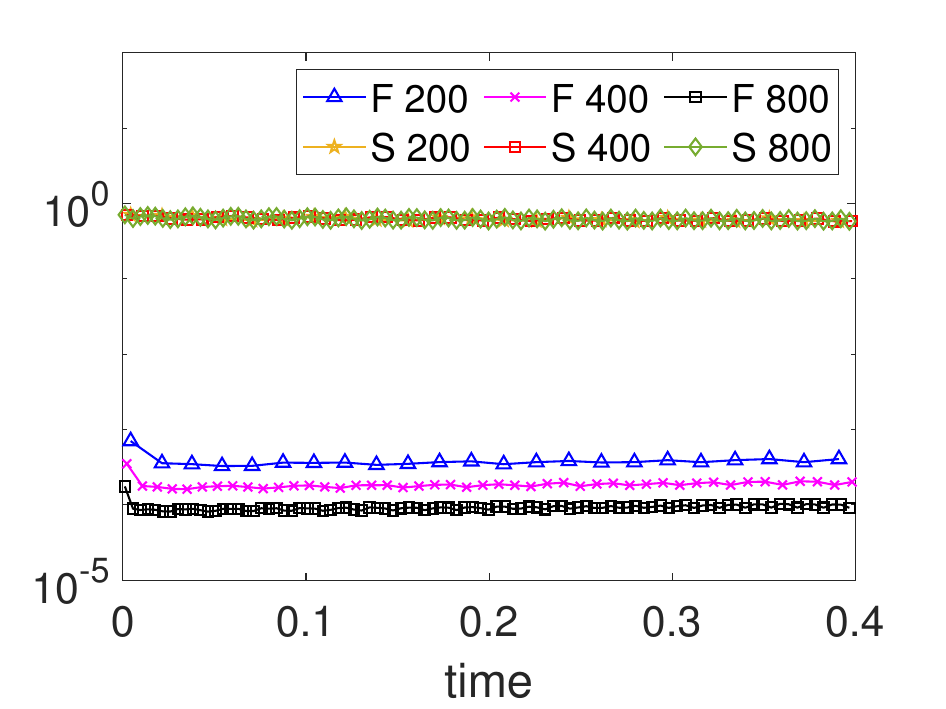}}
	\subfigure[Double rarefaction wave.]{ \includegraphics[width=0.45\linewidth]{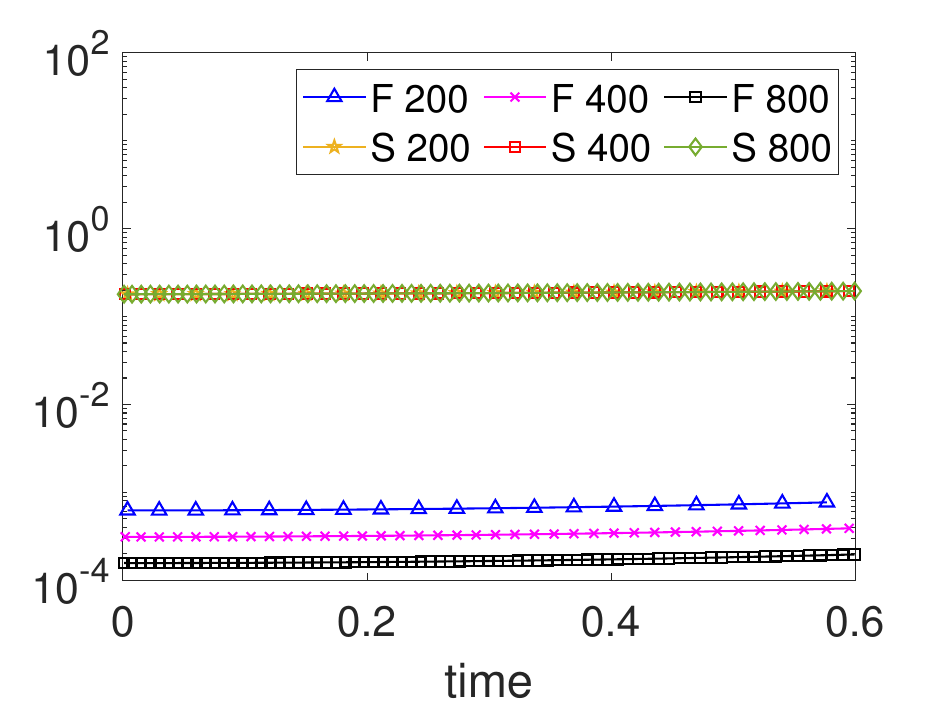}}
	\caption{Example \ref{ex:Sod} and Example \ref{ex:DRF}: One-dimensional Sod-like shock tube and double rarefaction wave. The maximum time step restriction by PP.
    }
   \label{figCFL}
\end{figure}

\subsection{Two-dimensional tests}

\begin{Ex}
\textbf{(Accuracy test.)}
\label{ex:acc2D}
\end{Ex}

In this part, we test the accuracy of the two-dimensional scheme with the gravitational potential $\phi_x=\phi_y=1$. The computational domain is taken as $\Omega=[0,2]\times[0,2]$, and the exact smooth solution is given by
$$\begin{aligned}
\rho &=1+0.2\sin \left( \pi(x+y-2t) \right),
\\u&=v=1,
\\ p&=5.5-x-y+2t+0.2\cos(\pi(x+y-2t))/\pi.\end{aligned}
$$
We run the simulation up to $T = 2$. The boundary conditions are set as the exact solutions when needed. We present the errors and orders of density for $k = 1,2,3$ in Table \ref{tabacc2D}, and the optimal convergence rates are always observed.

\begin{table}[htb!]
        \centering
        \caption{Example \ref{ex:acc2D}: Two-dimensional accuracy test. Errors and orders of density at final time $T = 2$.}
	\setlength{\tabcolsep}{3mm}{
		\begin{tabular}{|c|c|cc|cc|cc|}
			\hline 
   &$N$ & $L^1$ error & order & $L^2$ error & order & $L^\infty$ error & order  \\ \hline 
   \multirow{4}{*}{$k=1$}
     &      20 &9.37e-03 & -- &1.32e-02 & -- &8.66e-02 &-- \\
&      40 &2.26e-03 & 2.05 &3.15e-03 & 2.07 &2.55e-02 &1.76 \\
&      80 &5.62e-04 & 2.01 &7.70e-04 & 2.03 &6.99e-03 &1.87 \\
&     160 &1.43e-04 & 1.98 &1.92e-04 & 2.00 &1.75e-03 &2.00 \\
			\hline
   \multirow{4}{*}{$k=2$}
     &      20 &2.64e-04 & -- &3.48e-04 & -- &9.69e-04 &-- \\
&      40 &4.08e-05 & 2.70 &5.47e-05 & 2.67 &1.78e-04 &2.44 \\
&      80 &5.54e-06 & 2.88 &7.48e-06 & 2.87 &2.60e-05 &2.78 \\
&     160 &7.12e-07 & 2.96 &9.64e-07 & 2.96 &3.44e-06 &2.92 \\
    \hline
  \multirow{4}{*}{$k=3$}
		&      20 &1.64e-06 & -- &2.05e-06 & -- &7.09e-06 & -- \\
&      40 &9.97e-08 & 4.04 &1.25e-07 & 4.04 &3.89e-07 &4.19 \\
&      80 &6.19e-09 & 4.01 &7.73e-09 & 4.01 &2.44e-08 &4.00 \\
&     160 &3.87e-10 & 4.00 &4.83e-10 & 4.00 &1.58e-09 &3.95 \\
   \hline
	\end{tabular}} 
    \label{tabacc2D}
\end{table}

\begin{Ex}
\textbf{(Keplerian disk.)}
\label{ex:WB2D}
\end{Ex}

This example is first studied in \cite{gaburro2018well} under the polar coordinates $(r,\varphi)$, where $r=\sqrt{x^2+y^2}$ denotes the radius and $\varphi=\arctan(y/x)$ denotes the polar angle. It describes a rotating equilibrium state in a disk with the gravitational potential $\phi(r)=-1/r$. 
The computational domain is $\Omega=\{(x,y):1\le r\le 2\}$. In numerical simulation, we embedded the domain in a rectangle $[-2,2]\times[-2,2]$, and set the cell $K_{ij}\in\mathcal K$ if $1<r(x_i,y_j)<2$.  
 We consider the equilibrium state 
$$
\rho^e =1,\quad u^e=-\sqrt{\frac{1}{r}}\sin \varphi,\quad v^e=\sqrt{\frac{1}{r}}\cos \varphi,\quad p^e=1,
$$
which is a genuinely moving equilibrium state on Cartesian coordinates. 

First, we take the initial data $\mathbf U_h=\mathbf U_h^e$ and simulate until $T = 2$. The boundary conditions are set to be the initial values. The balance errors of density are given in Table \ref{tabWB2D} with $N_x=N_y=:N$, demonstrating that our
proposed WBESPP scheme can maintain the balance errors at the machine level for the 2D problem. 

\begin{table}[htb!]
        \centering
        \caption{Example \ref{ex:WB2D}: Keplerian disk. Errors and orders of density at $T = 2$.}
	\setlength{\tabcolsep}{3.2mm}{
		\begin{tabular}{|c|c|cc|cc|cc|}
			\hline 
   &$N$ & $L^1$ error & order & $L^2$ error & order & $L^\infty$ error & order  \\ \hline	
   \multirow{3}{*}{WBESPP}
   
    &      20 &1.73e-15 & -- &2.72e-15 & -- &1.49e-14 &-- \\
&      40 &3.41e-15 & -- &4.91e-15 & -- &1.62e-14 &-- \\
&      80 &7.05e-15 & -- &9.63e-15 & -- &2.69e-14 &-- \\

			\hline

    \multirow{3}{*}{non-WB}
   
    &      20 &7.89e-05 & -- &1.47e-04 & -- &1.32e-03 &-- \\
&      40 &1.12e-05 & 2.82 &2.10e-05 & 2.81 &2.04e-04 &2.70 \\
&      80 &1.45e-06 & 2.95 &2.82e-06 & 2.90 &2.93e-05 &2.80 \\

			\hline
            
	\end{tabular}} 
    \label{tabWB2D}
\end{table}

Next, a small perturbation is added on density initially
$$ \rho=\rho^e+10^{-6}\exp\left\{-50((x+1.5)^2+y^2)\right\}. $$
We use $N_x=N_y=120$ meshes and simulate until $T=2.5$. The boundary conditions are set to be the initial values. The expected result is the transport of this density perturbation at different velocities along polar-axis that are bigger at the interior and smaller at the exterior. In Fig \ref{figWB2D}, we compare the result of the density perturbation of non-WB scheme and the proposed WBESPP scheme. It can be seen that the non-WB scheme has large error, while the WBESPP scheme can capture the small perturbation accurately.

\begin{figure}[htbp!]
	\centering
    \subfigure[non-WB.]{
		\includegraphics[width=0.4\linewidth]{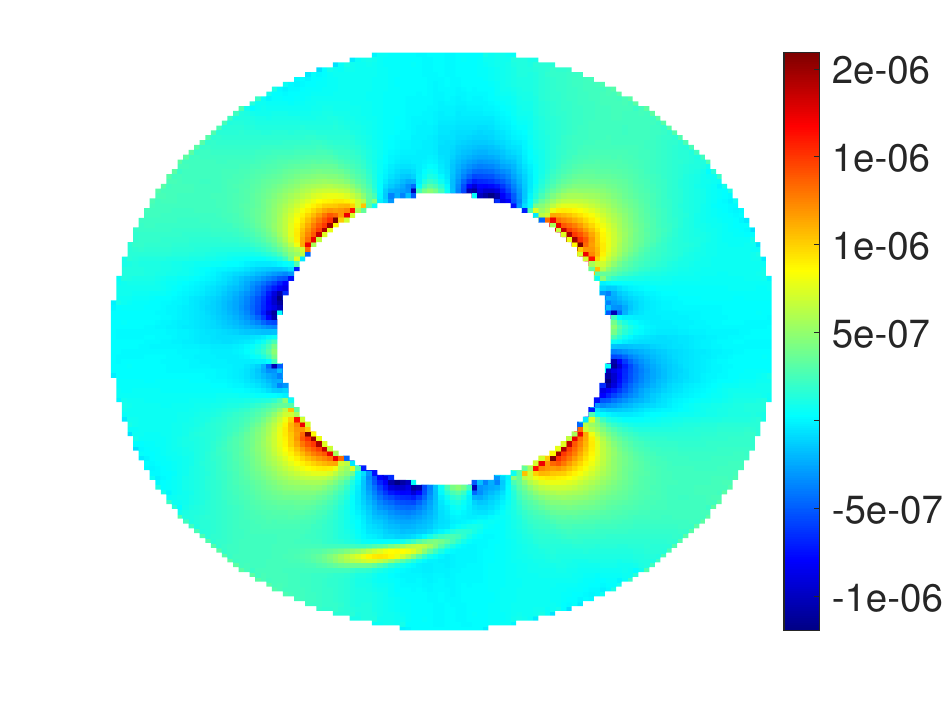}}
	\subfigure[WBESPP.]{
		\includegraphics[width=0.4\linewidth]{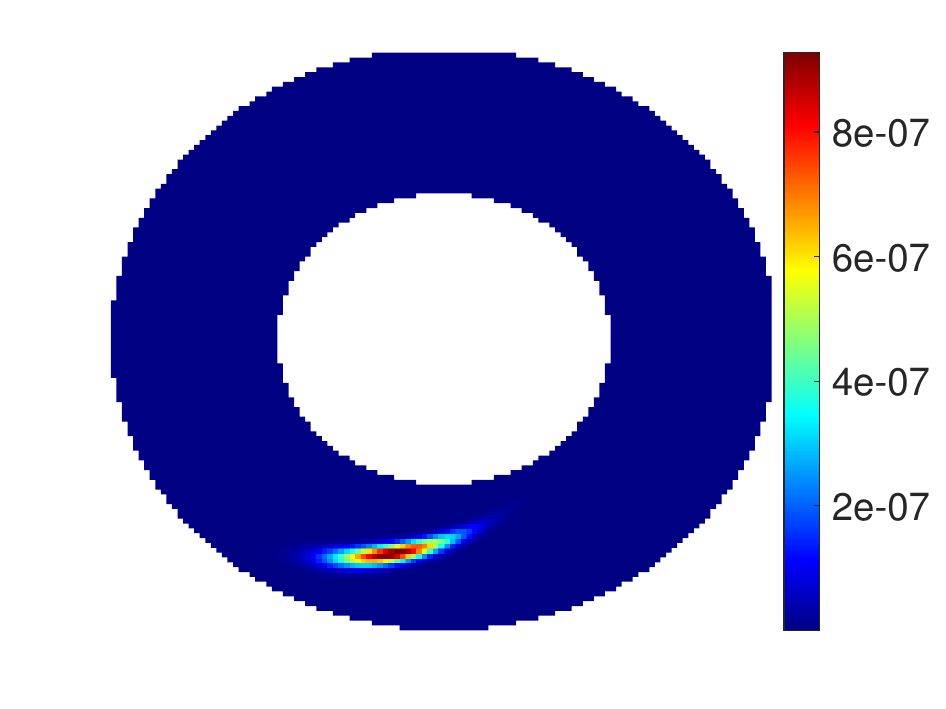}}
	\caption{Example \ref{ex:WB2D}: Keplerian disk. The result of density perturbation at $T=2$ with $N_x=N_y=120$.}
   \label{figWB2D}
\end{figure}

Furthermore, we simulate a more challenge case, where the density contains a discontinuity
$$
\rho ^e=\begin{cases}
	2,\ r<1.4,\\
	1,\ r\ge 1.4,\\
\end{cases}  u^e=-\sqrt{\frac{1}{r}}\sin \varphi,\quad v^e=\sqrt{\frac{1}{r}}\cos \varphi,\quad p^e=1.
$$
We simulate this case until $T=2$, and the balance errors are shown in Table \ref{tabWB2D2}. In Fig \ref{figWB2D2}, we present the density perturbation with $N_x=N_y=120$. It can be seen that the non-WB scheme shows large error since the discontinuity, while the WBESPP scheme preserves the error at round-off level as well.

\begin{table}[htb!]
        \centering
        \caption{Example \ref{ex:WB2D}: Keplerian disk with a discontinuity. Errors of density at $T = 2$.}
	\setlength{\tabcolsep}{3.2mm}{
		\begin{tabular}{|c|c|c|c|c|}
			\hline 
   &$N$ & $L^1$ error  & $L^2$ error  & $L^\infty$ error   \\ \hline	
   \multirow{3}{*}{WBESPP}
   
    &      20 &2.22e-15  &3.69e-15  &2.58e-14  \\
&      40 &4.42e-15  &6.99e-15  &5.75e-14  \\
&      80 &9.23e-15  &1.52e-14  &3.95e-13  \\

			\hline

    \multirow{3}{*}{non-WB}
   
    &      20 &4.25e-02  &1.14e-01 &6.54e-01 \\
&      40 &2.28e-02  &7.07e-02  &5.98e-01  \\
&      80 &1.32e-02 &5.44e-02 &5.99e-01  \\

			\hline
            
	\end{tabular}} 
    \label{tabWB2D2}
\end{table}

\begin{figure}[htbp!]
	\centering
    \subfigure[non-WB.]{
		\includegraphics[width=0.4\linewidth]{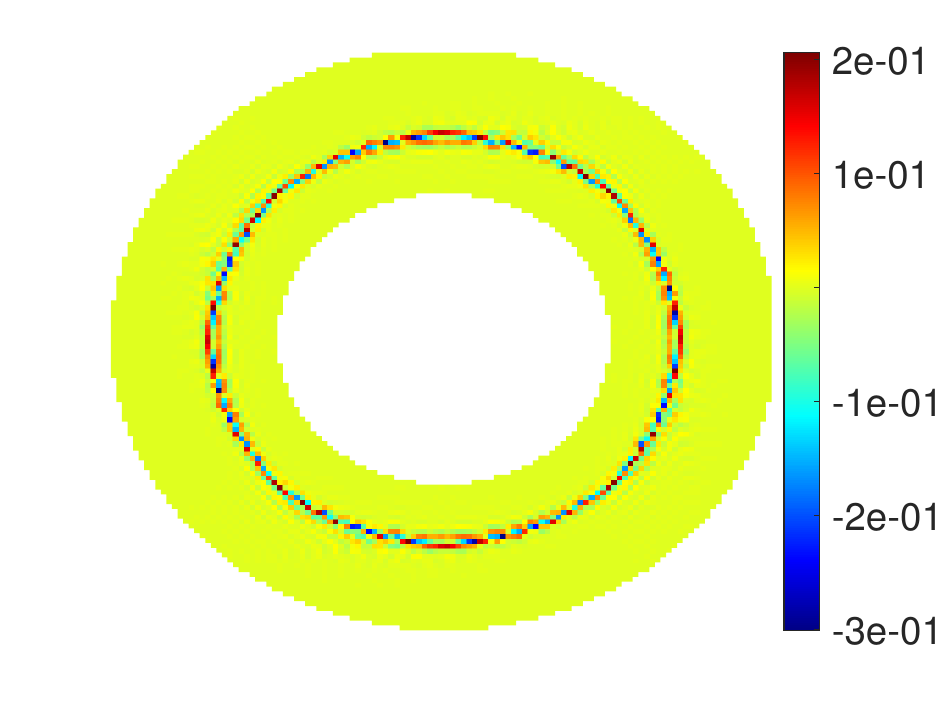}}
	\subfigure[WBESPP.]{
		\includegraphics[width=0.4\linewidth]{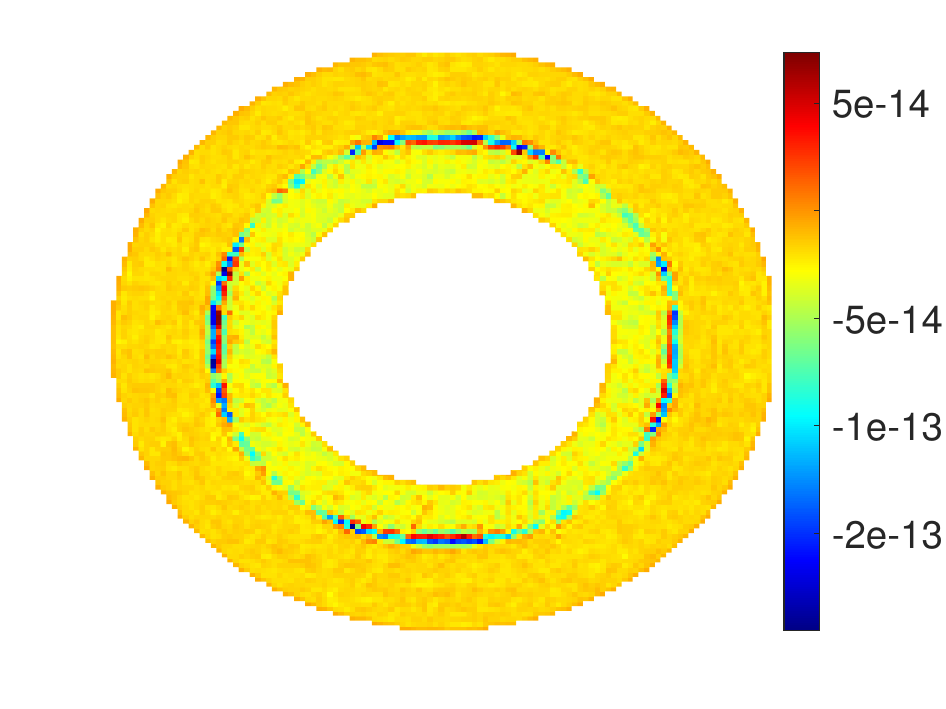}}
	\caption{Example \ref{ex:WB2D}: Keplerian disk with a discontinuity. The density perturbation with $N_x=N_y=120$ at $T=2$.}
   \label{figWB2D2}
\end{figure}

\begin{Ex}
\textbf{(Two-dimensional Riemann problem.)}
\label{ex:Riemann2D}
\end{Ex}

Here, we consider a Riemann problem on the above Keplerian disk \cite{gaburro2018well}.
The computational domain is $\Omega=\{(x,y):\ 1\le r\le 4\}$, and the initial data is given by
$$
(\rho,u,v,p)=\left\{\begin{array}{ll}
	(1,\ -(\sin\varphi)/\sqrt r,\ (\cos\varphi)/\sqrt r,\ 1),&r<2.5,
	 \\(0.1,\ -(\sin\varphi)/\sqrt r,\ (\cos\varphi)/\sqrt r,\ 0.1),&r\ge 2.5.
\end{array}\right.
$$
We use $N_x=N_y=500$ meshes to simulate this problem until $T=0.5$. The result of density, radial velocity $u_r=u\cos\varphi+v\sin\varphi$, and pressure are shown in Fig \ref{figRiemann2D}. Meanwhile, their 1D cuts along $\varphi=0$ are shown in Fig \ref{figRiemann2D2}. It can be seen that the structures of discontinuity are resolved well. We note that the non-ES scheme will blow up in the first several steps.

\begin{figure}[htbp!]
	\centering
    \subfigure[Density.]{
		\includegraphics[width=0.31\linewidth]{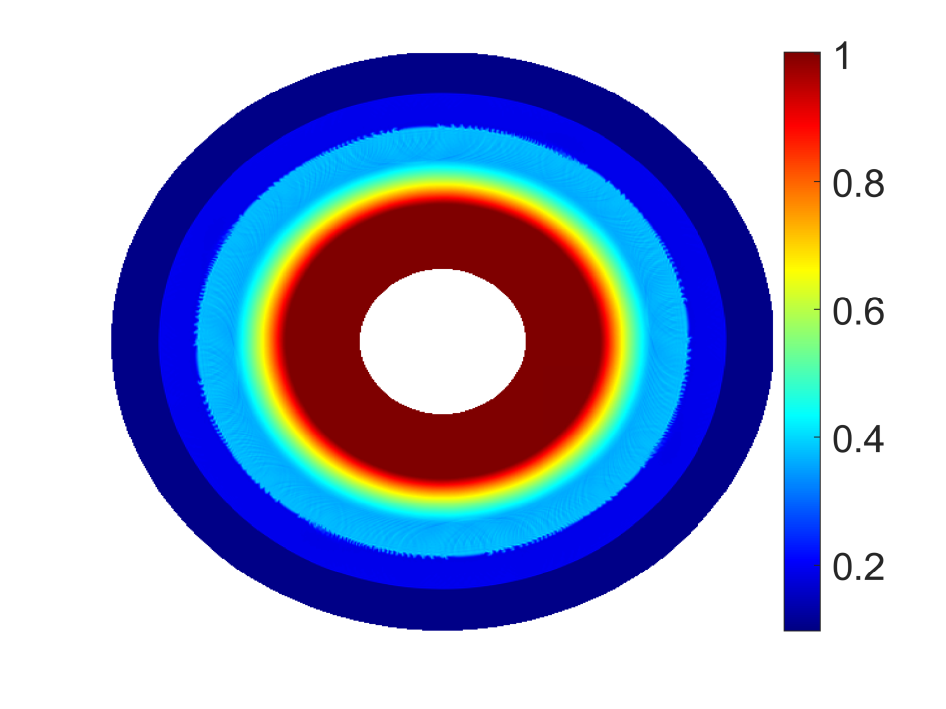}}
	\subfigure[$u_r=u\cos\varphi+v\sin\varphi.$]{
		\includegraphics[width=0.31\linewidth]{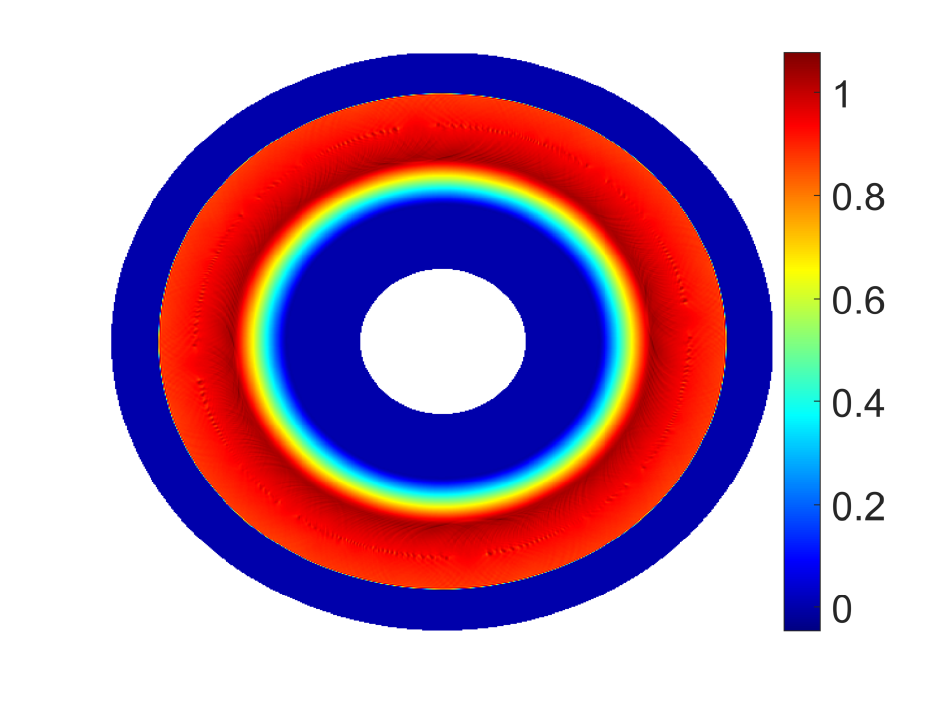}}
    \subfigure[Pressure.]{
		\includegraphics[width=0.31\linewidth]{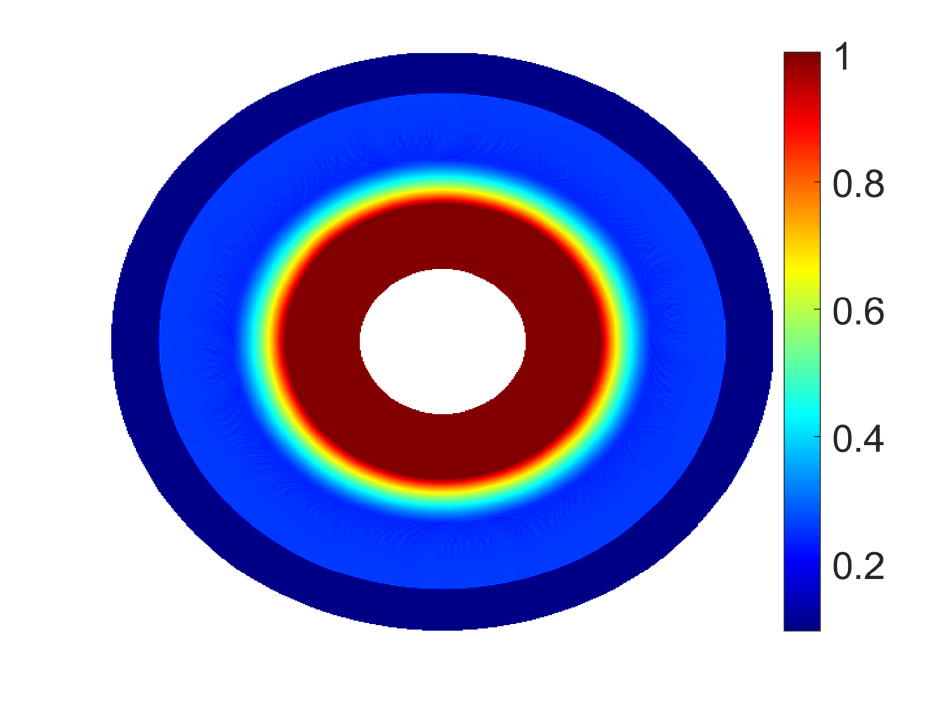}}
	\caption{Example \ref{ex:Riemann2D}: Two-dimensional Riemann problem on a Keplerian disk. The numerical solution at $T=0.5$ on $N_x=N_y=500$ meshes.}
   \label{figRiemann2D}
\end{figure}

\begin{figure}[htbp!]
	\centering
    \subfigure[Density.]{
		\includegraphics[width=0.31\linewidth]{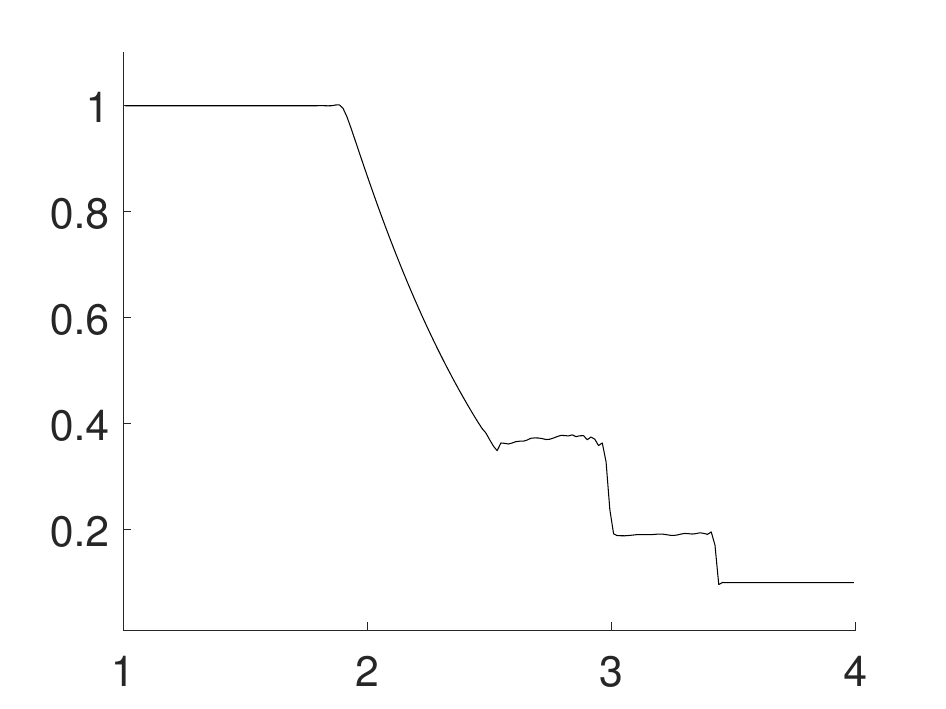}}
	\subfigure[$u_r=u\cos\varphi+v\sin\varphi.$]{
		\includegraphics[width=0.31\linewidth]{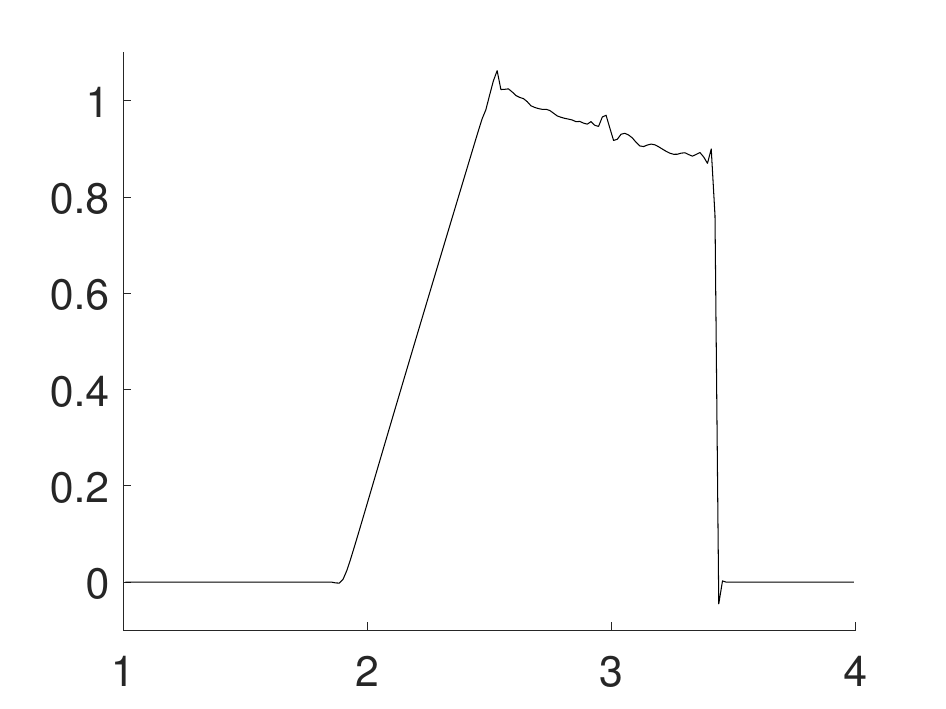}}
    \subfigure[Pressure.]{
		\includegraphics[width=0.31\linewidth]{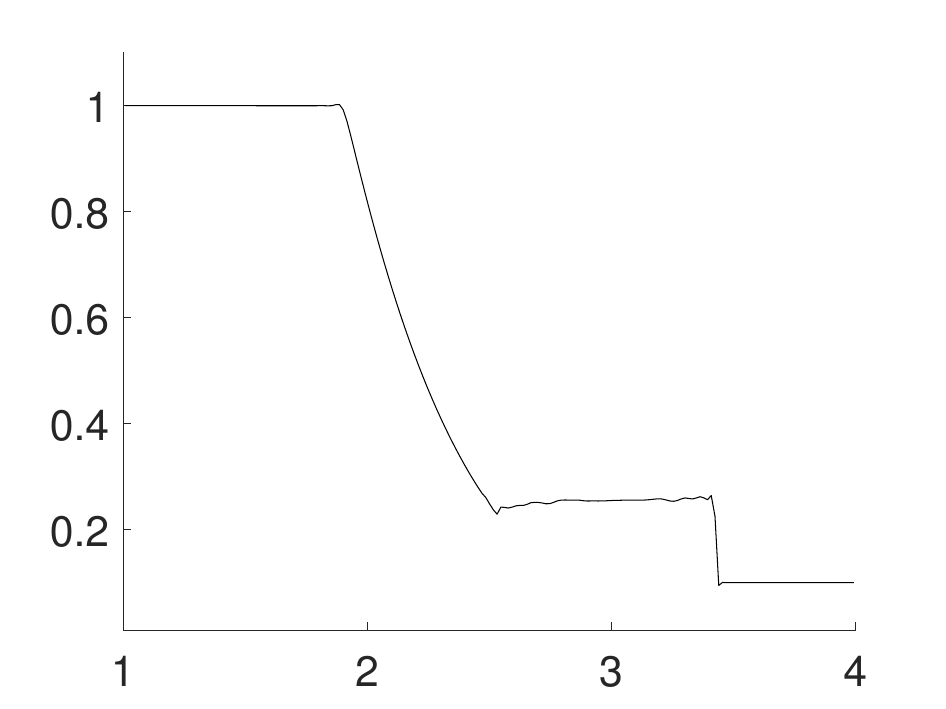}}
	\caption{Example \ref{ex:Riemann2D}: Two-dimensional Riemann problem on a Keplerian disk. The 1D cuts at $\varphi = 0$ on $N_x=N_y=500$ meshes.}
   \label{figRiemann2D2}
\end{figure}

\begin{Ex}
\textbf{(Kelvin–Helmholtz instability.)}
\label{ex:KH}
\end{Ex}

Here, we consider the Kelvin-Helmholtz instability on the above Keplerian disk, which is first studied in \cite{gaburro2018well}. For this example, the equilibrium state is set to
$$
\rho ^e=1+A_1\tanh \left( \frac{r-r_c}{\sigma _1} \right) ,\quad u^e=-\sqrt{\frac{1}{r}}\sin \varphi ,\quad v^e=\sqrt{\frac{1}{r}}\cos \varphi ,\quad p^e=1,
$$
and the initial data is the above equilibrium state with small perturbation
$$
\delta \rho =\Delta_0 ,\quad  
\delta u=\Delta_0 \cos \varphi ,\quad  
\delta v=\Delta_0\sin \varphi ,\quad  
\delta p=\Delta_0,
$$ 
where
$$ \Delta_0 =A_2\sin(\kappa\varphi)\exp\left\{ -\frac{(r-r_c)^2}{\sigma_2^2} \right\}.
$$
We set $A_1=0.25,\ A_2=0.5,\ \sigma_1=0.1,\ \sigma_2=0.005,\ r_c=1.5,\ \kappa=8$. We use a mesh with $N_x=N_y=240$ to simulate this instability until $T=17.5$. The results of density at $T=2.5,\ T=10,\ T=17.5$ are shown in Fig \ref{figKH}, it can be seen that the small structures are resolved well. This example demonstrates the capability of our scheme to maintain moving equilibrium solutions in astrophysical
simulations while accurately capturing small-scale physical instabilities.

\begin{figure}[htbp!]
	\centering
    \subfigure[$T = 2.5.$]{
		\includegraphics[width=0.31\linewidth]{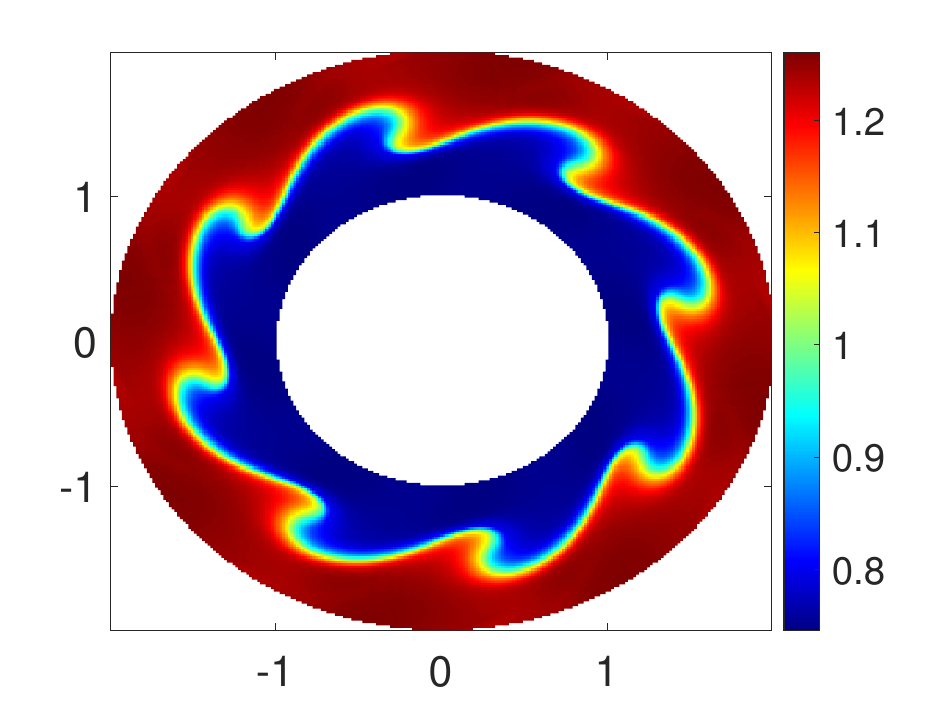}}
	\subfigure[$T = 10.$]{
		\includegraphics[width=0.31\linewidth]{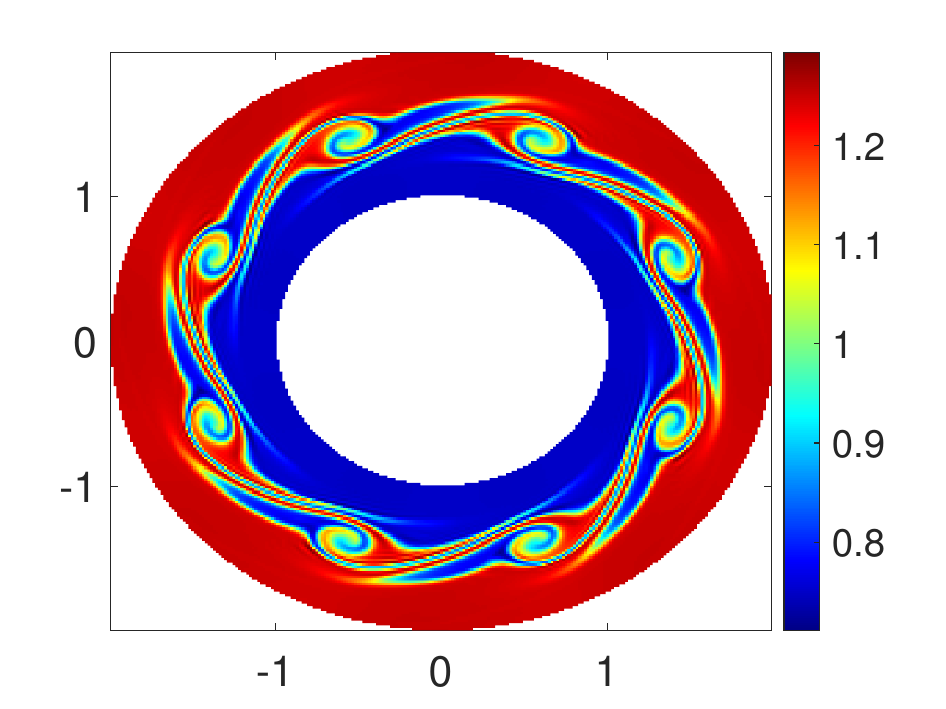}}
    \subfigure[$T = 17.5.$]{
		\includegraphics[width=0.31\linewidth]{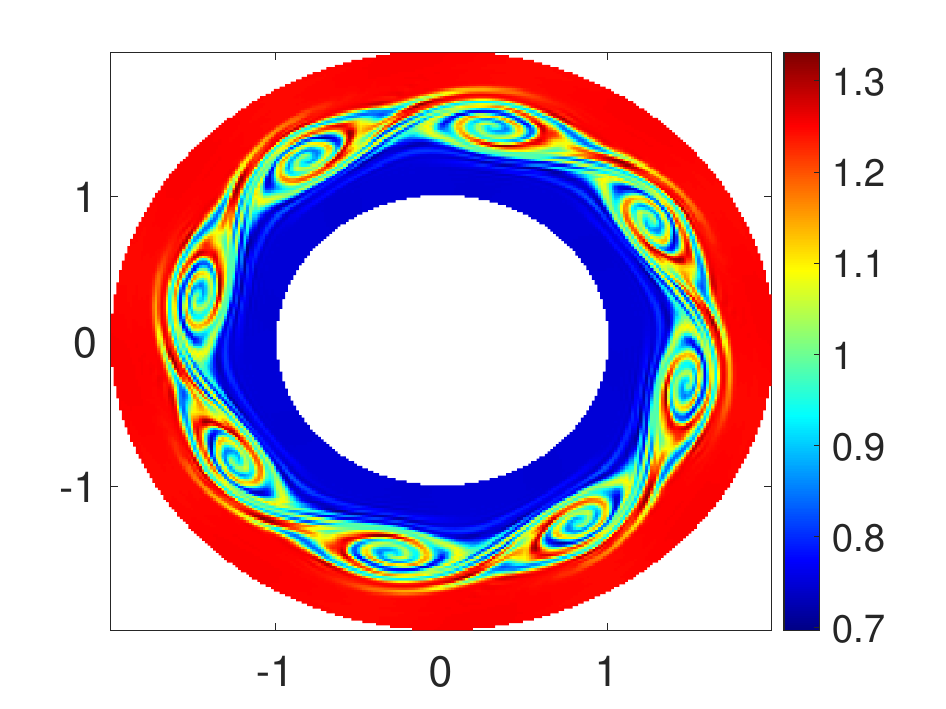}}
	\caption{Example \ref{ex:KH}: Kelvin–Helmholtz instability on a Keplerian disk. The numerical solution of density with $N_x=N_y=240$.}
   \label{figKH}
\end{figure}

\begin{Ex}
\textbf{(Double rarefaction wave.)}
\label{ex:DRF2D}
\end{Ex}

Finally, we examine the two-dimensional double rarefaction test case from \cite{jiang2022positivity}, which serves to validate the PP property of numerical schemes. We consider the computational domain $\Omega=[-0.5,0.5]\times[-0.5,0.5]$ with outflow boundaries. The gravitational potential is set as $\phi=0.5(x^2+y^2)$, and the initial conditions are specified as follows:
$$
\rho =\exp \left( -\phi \left( x,y \right) /0.4 \right) ,\quad p=0.4\exp \left( -\phi \left( x,y \right) /0.4 \right), 
$$
$$
u=\left\{\begin{array}{rr}
	-2,&x\le 0,
	 \\2,&x>0,
\end{array}\right. \quad v=0.
$$
The simulation is conducted until $T=0.1$ on $N_x\times N_y=200\times 200$ meshes, with results presented in Fig \ref{figDRF2D}. The numerical results demonstrate that the proposed WBESPP scheme maintains both stability and accuracy throughout the computation. It should be noted that the conventional non-PP scheme exhibits numerical instability and will blow up in the first several steps.

\begin{figure}[htbp!]
	\centering
    \subfigure[Density.]{
		\includegraphics[width=0.31\linewidth]{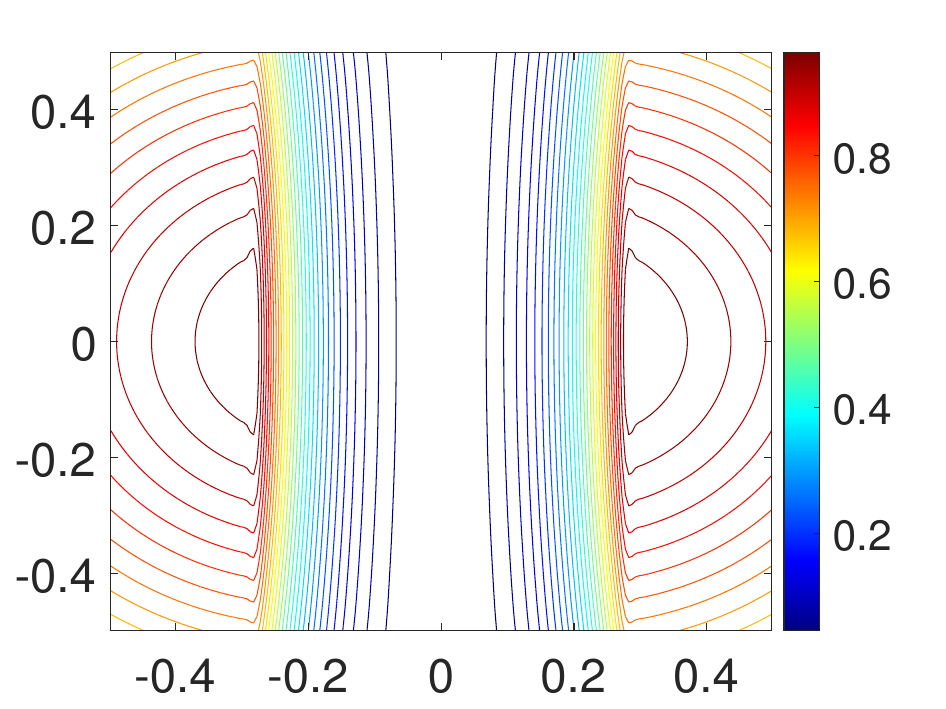}}
	\subfigure[$m=\rho u$.]{
		\includegraphics[width=0.31\linewidth]{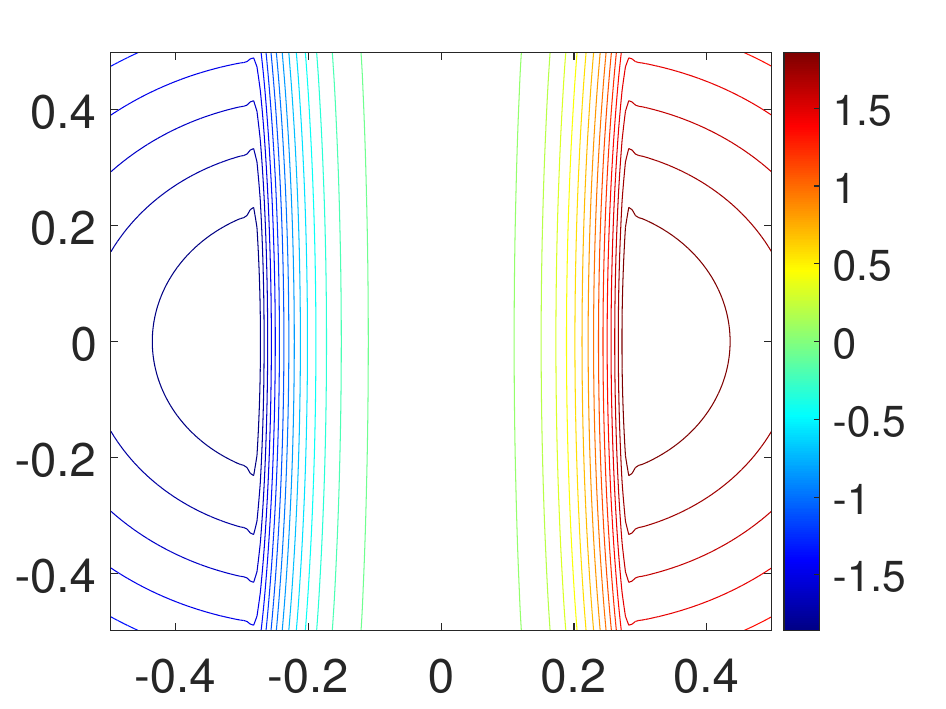}}
    \subfigure[Pressure.]{
		\includegraphics[width=0.31\linewidth]{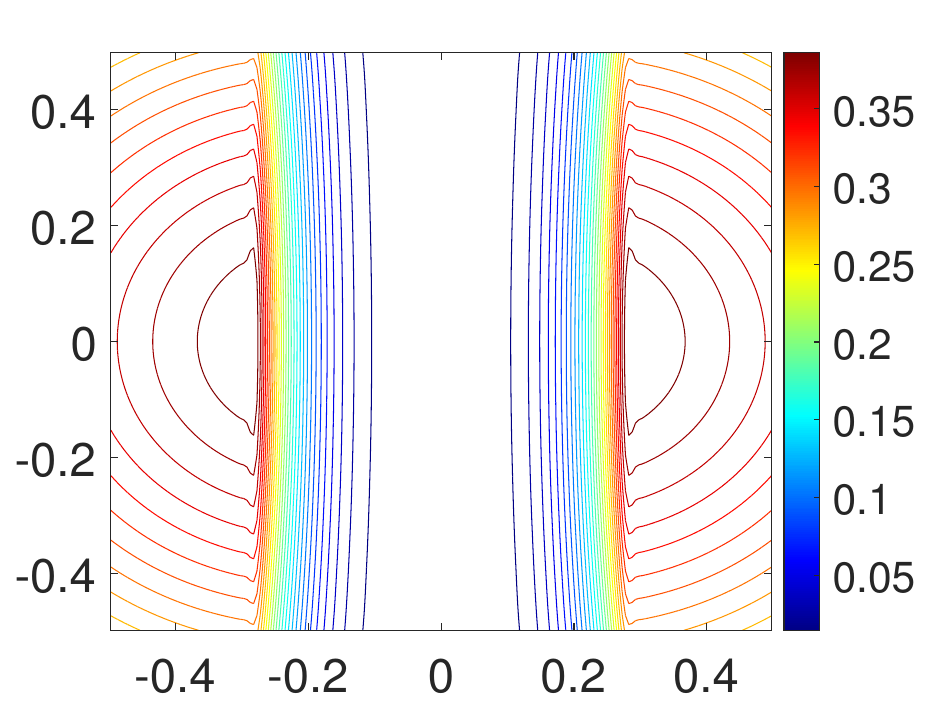}}
	\caption{Example \ref{ex:DRF2D}: Two-dimensional double rarefaction wave problem. The numerical solution with $N_x=N_y=200$ at $T=0.1$.}
   \label{figDRF2D}
\end{figure}

\section{Concluding remarks}\label{sec6}

In this paper, we have developed a novel structure-preserving nodal DG method for the Euler equations with gravity. The proposed scheme simultaneously achieves three critical properties: well-balancedness for general equilibrium states, entropy stability, and positivity-preservation. Through rigorous theoretical analysis, we have proven that the method maintains these properties while achieving high-order accuracy for smooth solutions. Extensive numerical experiments have demonstrated the scheme's effectiveness in handling complex flow features while preserving the desired physical properties. The significance of this work lies in its unified treatment of multiple physical constraints without compromising accuracy, making it particularly suitable for long-time simulations in astrophysics and atmospheric applications. Future work could focus on extending this framework to unstructured meshes and other equations.

\bibliographystyle{abbrv}
\bibliography{tex}

\end{document}